\documentclass[reqno]{amsart}
\usepackage{amsmath}
\usepackage{amsthm}
\usepackage{amssymb}
\usepackage{mathrsfs} 
\usepackage{bm}
\usepackage{hyperref}
\usepackage{comment}
\usepackage{tikz}
\usetikzlibrary{arrows.meta}
\tikzset{
  .../.tip={[sep=2pt 1]
    Round Cap[]. Circle[length=0.5pt 1] Circle[length=0.5pt 1] Circle[length=0.5pt 1, sep=2pt]}}
    
\usepackage{pgffor}
\usetikzlibrary{shapes.misc}
\usepackage{pgflibraryshapes}
\usetikzlibrary{decorations.pathreplacing,calc,arrows.meta}
\tikzset{%
  half dotted/.style={
    decoration={show path construction, 
      lineto code={
          \draw[#1] (\tikzinputsegmentfirst) --($(\tikzinputsegmentfirst)!.3!(\tikzinputsegmentlast)$);,
          \draw[loosely dotted,-,#1] ($(\tikzinputsegmentfirst)!.5!(\tikzinputsegmentlast)$)--(\tikzinputsegmentlast);,
      }
    },
    decorate
  },
}

\newcommand{\REG}{{\rm REG}}
\newcommand{\SING}{{\rm SING}}

\newcommand{\GCH}{{\rm GCH}}

\newcommand{\ZFC}{{\rm ZFC}}

\newcommand{\ORD}{\mathop{{\rm ORD}}}

\renewcommand{\emptyset}{\varnothing}

\renewcommand{\P}{{\mathbb P}}

\newcommand{\Q}{{\mathbb Q}}

\newcommand{\R}{{\mathcal R}}
\newcommand{\I}{{\mathcal I}}

\newcommand{\Ult}{\mathop{\rm Ult}}
\newcommand{\forces}{\Vdash}

\newcommand{\restrict}{\upharpoonright}
\newcommand{\concat}{\mathbin{{}^\smallfrown}}

\newcommand{\<}{\langle}
\renewcommand{\>}{\rangle}

\newcommand{\st}{\mid}

\newcommand{\defn}{\mathop{\rm def}}

\newcommand{\ran}{\mathop{\rm ran}}

\newcommand{\ot}{\mathop{\rm ot}\nolimits}

\newcommand{\id}{\mathop{\rm id}}

\newcommand{\crit}{\mathop{\rm crit}}

\newcommand{\Lev}{\mathop{\rm Lev}}

\newcommand{\gp}{\mathop{\rm \Gamma}}

\newcommand{\NS}{{\mathop{\rm NS}}}

\newcommand{\Tr}{{\mathop{\rm Tr}}}

\newcommand{\Refl}{{\mathop{\rm Refl}}}

\newcommand{\NSS}{{\mathop{\rm NSS}}}

\newcommand{\NSub}{{\mathop{\rm NSub}}}

\newcommand{\NPreRam}{{\mathop{\rm NPreRam}}}

\newcommand{\diagonalunion}{\bigtriangledown}

\renewcommand{\and}{\mathop{\&}}

\newtheorem{theorem}{Theorem}[section]
\newtheorem{lemma}[theorem]{Lemma}
\newtheorem{corollary}[theorem]{Corollary}
\newtheorem{proposition}[theorem]{Proposition}

\theoremstyle{definition}
\newtheorem{question}[theorem]{Question}
\newtheorem{remark}[theorem]{Remark}

\newtheorem{definition}[theorem]{Definition}

\thanks{The author would like to thank Sean Cox, Monroe Eskew, Victoria Gitman and Chris Lambie-Hanson for many helpful conversations regarding the topics of this article. The author also thanks the anonymous referee for the detailed review which greatly improved this article.}
\date{\today}

\begin{document}

\title{Large cardinal ideals}

\author[Brent Cody]{Brent Cody}
\address[Brent Cody]{ 
Virginia Commonwealth University,
Department of Mathematics and Applied Mathematics,
1015 Floyd Avenue, PO Box 842014, Richmond, Virginia 23284, United States
} 
\email[B. ~Cody]{bmcody@vcu.edu} 
\urladdr{http://www.people.vcu.edu/~bmcody/}

\begin{abstract}
Building on work of Holy, L\"ucke and Njegomir \cite{MR3913154} on small embedding characterizations of large cardinals, we use some classical results of Baumgartner (see \cite{MR0384553} and \cite{MR0540770}), to give characterizations of several well-known large cardinal ideals, including the Ramsey ideal, in terms of generic elementary embeddings; we also point out some seemingly inherent differences between small embedding and generic embedding characterizations of subtle cardinals. Additionally, we present a simple and uniform proof which shows that, when $\kappa$ is weakly compact, many large cardinal ideals on $\kappa$ are nowhere $\kappa$-saturated. Lastly, we survey some recent consistency results concerning the weakly compact ideal as well as some recent results on the subtle, ineffable and $\Pi^1_1$-indescribable ideals on $P_\kappa\lambda$, and we close with a list of open questions.
\end{abstract}



\maketitle

\tableofcontents



\section{Introduction}\label{section_introduction}

Baumgartner (see \cite{MR0384553} and \cite{MR0540770}) showed that many large cardinal properties can also be viewed as properties of subsets of cardinals and not just of the cardinals themselves, and this leads naturally to a consideration of ideals associated to large cardinals. For example, a set $X\subseteq\kappa$ is \emph{Ramsey} if for every function $f:[X]^{<\omega}\to \kappa$ with $f(a)<\min(a)$ for all $a\in[X]^{<\omega}$, there is a set $H\subseteq X$ of size $\kappa$ which is \emph{homogeneous} for $f$, meaning that $f\restrict[H]^n$ is constant for all $n<\omega$. Baumgartner showed that if $\kappa$ is a Ramsey cardinal then the collection of non-Ramsey subsets of $\kappa$ is a nontrivial normal ideal on $\kappa$. Similarly, one can define normal ideals associated to indescribability, subtlety, ineffability and many other large cardinal notions. In fact, Baumgartner showed that certain characterizations of almost ineffability, ineffability and Ramseyness require the consideration of large cardinal ideals (for example, see Remark \ref{remark_baumgartners_characterizations_of_ineffability} below).

It is a well-known and often-used fact that the stationarity of a set $S\subseteq\kappa$ can be characterized in terms of certain types of elementary embeddings. For example, a set $S\subseteq\kappa$ is stationary if and only if there is some forcing $\P$ such that whenever $G\subseteq\P$ is generic there is, in $V[G]$, an elementary embedding 
\[j:(V,\in)\to (M,\in^M)\subseteq V[G]\] with critical point $\kappa$ where $M$ is well-founded up to $\kappa^+$ such that $\kappa\in j(S)$. Recall that such elementary embeddings are obtained from stationary sets by using generic ultrapowers. Alternatively, $S\subseteq\kappa$ is stationary if and only if there is a nontrivial elementary embedding $j:M\to H(\kappa^+)$ such that $M$ is transitive of size less than $\kappa$, $j(\crit(j))=\kappa$, $S\in\ran(j)$ and $\crit(j)\in S$. Recall that such embeddings are obtained from stationary sets by taking the inverses of certain transitive collapse maps. Inspired by the work of Holy-L\"ucke-Njegomir \cite{MR3913154}, in this article we address the question: to what extent can these elementary embedding characterizations of the nonstationary ideal be generalized to large cardinal ideals? 

In Section \ref{section_preliminaries}, we cover some preliminaries, including a few basic properties of ideals we will need later on, as well as some facts concerning elementary embedding characterizations of stationary sets.

In Section \ref{section_embeddings}, we show that both the generic embedding and the transitive collapse embedding characterizations of stationarity can be generalized to many large cardinal ideals including the subtle, ineffable and Ramsey ideals. Let us emphasize that for some large cardinal notions, such as subtlety, inherent differences emerge between generic embedding and transitive collapse embedding characterizations (see Proposition \ref{proposition_subtle_embedding}, Remark \ref{remark_subtle_difference} and Proposition \ref{remark_subtle_difference} below).

In Section \ref{section_splitting}, we provide a simple and uniform proof of a folklore result, which states that Solovay's splitting theorem for stationary sets can be generalized to many large cardinal ideals on $\kappa$ including the $\Pi^1_n$-indescribable ideal, the almost ineffable ideal, the ineffable ideal, the Ramsey ideal and others, when these ideals are nontrivial. Specifically, we show that if $\kappa$ is weakly compact and $I$ is a normal ideal on $\kappa$ which is definable over $H(\kappa^+)$, then $I$ is nowhere $\kappa$-saturated.\footnote{The author would like to thank Sean Cox for pointing out this result and the included proof.} Let us emphasize that this result is essentially folklore, although it may not have been widely known previously since Hellsten states the result only for the $\Pi^1_n$-indescribable ideals \cite[Theorem 2]{MR2653962} and Foreman gives a different proof for the weakly compact ideal which uses generic embeddings \cite[Proposition 6.4]{MR2768692}.

In Section \ref{section_consistency_results}, we give a survey of some consistency results concerning the weakly compact ideal. We discuss a result of Hellsten \cite{MR2653962} on the saturation of the weakly compact ideal, as well as some results due to the author \cite{MR3985624} and Cody-Sakai \cite{MR4050036} on the weakly compact reflection principle. We also state a few theorems due to Gitman, Cody and Lambie-Hanson \cite{CodyGitmanLambieHanson} regarding forcing a $\square(\kappa)$-like principle to hold at a weakly compact cardinal.

In Section \ref{section_p_kappa_lambda}, we give a survey of selected results involving large cardinal ideals on $P_\kappa\lambda$. For example, we discuss two-cardinal versions of indescribability, subtlety and ineffability.

Finally, in Section \ref{section_questions}, we state several open questions.

\section{Preliminaries}\label{section_preliminaries}

\subsection{Basic terminology and facts about ideals}

An \emph{ideal} $I$ on a cardinal $\kappa$ is a collection of subsets of $\kappa$ which is closed under finite unions and closed under subsets. If $\kappa$ is a cardinal and $I$ is an ideal on $\kappa$ then $I^+=\{X\subseteq\kappa\st X\notin I\}$ is the collection of $I$-positive sets, $I^*=\{X\subseteq\kappa\st\kappa\setminus X\in I\}$ is the filter dual to $I$. If $S\in I^+$ then $I\restrict S=\{X\subseteq\kappa\st X\cap S\in I\}$ is an ideal on $\kappa$ extending $I$ and notice that $S\in (I\restrict S)^*$. An ideal $I$ on $\kappa$ is \emph{normal} if the collection $I^+$ satisfies the Fodor property: for all $S\in I^+$ whenever $f:S\to\kappa$ satisfies $f(\alpha)<\alpha$ for all $\alpha\in S$, then there is an $H\in P(S)\cap I^+$ such that $f\restrict H$ is constant. Equivalently, $I$ is normal if it is closed under diagonal unions, that is, whenever $\{X_\alpha\st\alpha<\kappa\}\subseteq I$ we have $\bigtriangledown_{\alpha<\kappa}X_\alpha:=\{\beta<\kappa\st \beta\in\bigcup_{\alpha<\beta}X_\alpha\}\in I$.

In Section \ref{section_embeddings} below, we will be concerned with showing that certain large cardinal ideals are obtained by taking the ideal generated by a union of some other large cardinal ideals. Given a family $\mathcal{A}\subseteq P(\kappa)$ of subsets of $\kappa$, the \emph{ideal generated by $\mathcal{A}$} is defined to be
\[\overline{A}=\{X\subseteq\kappa\st (\exists\mathcal{B}\in[\mathcal{A}]^{<\omega}) X\subseteq\bigcup\mathcal{B}\}.\]
the collection of all subsets of $\kappa$ which are contained in some union of finitely many sets from $\mathcal{A}$. In what follows we will repeatedly use both of the following observations.

\begin{remark}
It is easy to see that if $J$ and $K$ are ideals on a cardinal $\kappa$, then the ideal on $\kappa$ generated by $J\cup K$ is the collection of all $X\subseteq\kappa$ which can be written as a disjoint union of a set from $J$ and a set from $K$. That is, 
\[\overline{J\cup K}=\{X\subseteq \kappa\st(\exists A\in J)(\exists B\in K)(X=A\cup B \text{ and } A\cap B=\emptyset)\}.\]
\end{remark}

\begin{remark}\label{remark_ideal_containment}
Suppose $I_0$, $I_1$ and $J$ are ideals on $\kappa$. If we want to prove that $J=\overline{I_0\cup I_1}$, part of what we must show is that $J\supseteq\overline{I_0\cup I_1}$, or equivalently $J^+\subseteq\overline{I_0\cup I_1}^+$. Notice that we may obtain a chain of equivalences directly from the definitions involved:
\begin{align*}
J^+\subseteq\overline{I_0\cup I_1}^+ &\iff \overline{I_0\cup I_1}\subseteq J \\
	&\iff I_0\cup I_1\subseteq J\\
	&\iff J^+\subseteq I_0^+\cap I_1^+.
\end{align*}
In what follows, in order to prove that the property $J^+\subseteq\overline{I_0\cup I_1}^+$ (or equivalently the property $J\supseteq\overline{I_0\cup I_1}$) holds for various ideals, we will prove $J^+\subseteq I_0^+\cap I_1^+$ and include a reference to this remark.
\end{remark}

Given an ideal $I$ on $\kappa$, we write $P(\kappa)/I$ to denote the usual atomless\footnote{We write $P(\kappa)/I$ when we really mean $P(\kappa)/I-\{[\emptyset]\}$.} boolean algebra obtained from $I$. If $G$ is $(V,P(\kappa)/I)$-generic, then we let $U_G$ be the canonical $V$-ultrafilter obtained from $G$ extending the dual filter $I^*$. The appropriate version of {\L}os's Theorem can be easily verified, and thus we obtain a canonical generic ultrapower embedding $j:V\to V^\kappa/U_G$ in $V[G]$. If $I$ is a normal ideal then the generic ultrafilter $U_G$ is $V$-normal and the critical point of the corresponding, possibly illfounded, generic ultrapower $j:V\to V^\kappa/U_G\subseteq V[G]$ is $\kappa$. The following lemma is a standard tool for working with generic ultrapowers (see \cite[Lemma 22.14]{Jech:Book} or \cite[Section 2]{MR2768692} for more details).

\begin{lemma}
Suppose $I$ is a normal ideal on $\kappa$. Let $G$ be $(V,P(\kappa)/I)$-generic and $j:V\to V^\kappa/U_G$ be the corresponding generic elementary embedding and let $E$ be the ultrapower of the $\in$ relation. Then
\begin{enumerate}
\item $E$ is wellfounded on the ordinals up to $\kappa^+$,
\item $\kappa=[id]_{U_G}$,
\item for all $X\in P(\kappa)^V$ we have $X\in U_G$ if and only if $\kappa E j(X)$ and 
\item for all $f:\kappa\to V$ with $f\in V$ we have $[f]_{U_G}=j(f)(\kappa)$.
\end{enumerate}
\end{lemma}

\begin{definition}
When we say \emph{there is a generic elementary embedding $j:V\to M\subseteq V[G]$} we mean that there is some forcing poset $\P$ such that whenever $G$ is $(V,\P)$-generic then, in $V[G]$, there are definable classes $M$, $E$ and $j$ such that $j:(V,\in)\to (M,E)$ is an elementary embedding, where $(M,E)$ is possibly not wellfounded.
\end{definition}

\subsection{Small embeddings}

Suppose $\kappa$ is a regular uncountable cardinal. By iteratively taking Skolem hulls, one can build an elementary substructure $X\prec H(\kappa^+)$ such that $\kappa\in X$, $X\cap\kappa\in \kappa$ and $|X|<\kappa$. Let $\pi: X\to M$ be the transitive collapse of $X$ and let $j=\pi^{-1}$. Then $j:M\to H(\kappa^+)$ is an elementary embedding with $j(\crit(j))=\kappa$. Holy, L\"ucke and Njegomir \cite{MR3913154} used such embeddings to give new characterizations of several large cardinal notions, including subtlety, $\Pi^1_n$-indescribability, ineffability, measurability and $\lambda$-supercompactness, and to give new proofs of results of Christoph Weiss \cite{MR2959668} on the consistency strength of certain generalized tree properties.


\begin{definition}[Holy, L\"ucke, Njegomir \cite{MR3913154}]
Given cardinals $\kappa<\theta$, we say that a non-trivial elementary embedding $j:M\to H(\theta)$ is a \emph{small embedding for $\kappa$} if $M\in H(\theta)$ is transitive and $j(\crit(j))=\kappa$.
\end{definition}

\subsection{Embedding characterizations of stationarity}

In this section we give two (folklore) characterizations of stationary subsets of a cardinal $\kappa$ in terms of elementary embeddings, and we discuss several modest generalizations of some results of \cite{MR3913154}.

\begin{proposition}[Folklore]\label{proposition_stationarity_characterizations}
Suppose $\kappa>\omega$ is a regular cardinal. The following are equivalent.
\begin{enumerate}
\item $S\subseteq\kappa$ is stationary.
\item There is a generic elementary embedding $j:V\to M\subseteq V[G]$ with critical point $\kappa$ such that $\kappa \in j(S)$.
\item There is a small embedding $j:M\to H(\kappa^+)$ for $\kappa$ such that $S\in\ran(j)$ and $\crit(j)\in S$.
\end{enumerate}
\end{proposition}

\begin{proof}
Let us show that (1) and (2) are equivalent. If $S\subseteq\kappa$ is stationary, let $G$ be $(V,P(\kappa)/(\NS_\kappa\restrict S))$-generic, let $U_G$ be the generic ultrafilter obtained from $G$ and let $j:V\to M=V^\kappa/U_G$ be the corresponding generic ultrapower embedding. Since $U_G$ is a $V$-normal $V$-ultrafilter, the critical point of $j$ is $\kappa$ and since $S\in U_G$ we have $\kappa \in j(S)$. Conversely, suppose there is a generic elementary embedding $j:V\to M\subseteq V[G]$ with critical point $\kappa$ such that $\kappa \in j(S)$. If $C\subseteq\kappa$ is a club in $V$, then $\kappa \in j(S\cap C)$ and by elementarity $S\cap C\neq\emptyset$.

Next we show that (1) and (3) are equivalent.  Suppose $S\subseteq\kappa$ is stationary. Let $\<X_\alpha\st\alpha<\kappa\>$ be a continuous increasing elementary chain of submodels of $H(\kappa^+)$ each of cardinality less than $\kappa$ such that $S\in X_0$, $\alpha\subseteq X_\alpha\cap\kappa\in \kappa$ for all $\alpha<\kappa$. Since $S$ is stationary, there is an $\alpha<\kappa$ such that $\alpha=X_\alpha\cap\kappa\in S$. Let $j:M\to H(\kappa^+)$ be the inverse of the Mostowski collapse of $X_\alpha$. Then $j$ is a small embedding for $\kappa$ with $S\in \ran(j)$ and $\crit(j)=X_\alpha\cap\kappa=\alpha$. Conversely, suppose $j:M\to H(\kappa^+)$ is a small embedding for $\kappa$ such that $S\in\ran(j)$ and $\crit(j)\in S$. Assume $S$ is not stationary in $\kappa$. Then $S$ is not a stationary subset of $\kappa$ in $H(\kappa^+)$ and by elementarity it follows that in $M$ there is a club $C\subseteq\crit(j)$ such that $C\cap j^{-1}(S)=\emptyset$. Again by elementarity, $j(C)\cap S=\emptyset$, but this is impossible since $\crit(j)\in j(C)\cap S$.
\end{proof}

\begin{proposition}[Folklore]\label{proposition_formula}
Given an $\mathcal{L}_\in$-formula $\varphi(v_0,v_1)$, the following statements are equivalent for every cardinal $\kappa$ and every set $x$.
\begin{enumerate}
\item $\kappa$ is regular and uncountable and $\{\alpha<\kappa\st \varphi(\alpha,x)\}$ is stationary in $\kappa$.
\item There is a generic elementary embedding $j:V\to M\subseteq V[G]$ with critical point $\kappa$ such that $M\models$ $\varphi(\kappa,j(x))$.
\item There is a small embedding $ j:M\to H(\kappa^+)$ for $\kappa$ such that $\varphi(\crit( j),x)$ holds and $x\in \ran( j)$.
\end{enumerate}
\end{proposition}

\begin{proof}
For the equivalence of (1) and (3) see \cite[Lemma 2.1]{MR3913154}. Let us show that (1) and (2) are equivalent. 

Let $S=\{\alpha<\kappa\st\varphi(\alpha,x)\}$. Since $S$ is stationary in $\kappa$ we may let $G\subseteq P(\kappa)/(\NS_\kappa\restrict S)$ be generic and let $U_G$ be the $V$-normal $V$-ultrafilter obtained from $G$. Let $j:V\to V^\kappa/U_G\subseteq V[G]$ be the corresponding generic ultrapower. Since $\NS_\kappa\restrict S$ is a normal ideal on $\kappa$, the critical point of $j$ is $\kappa$, and since $S\in (\NS_\kappa\restrict S)^*\subseteq U_G$ we have $\kappa \in j(S)$. In other words, $V^\kappa/U_G\models$ $\varphi(\kappa,j(x))$. For the converse, fix a club $C\in P(\kappa)^V$. Let $j:V\to M\subseteq M[G]$ be a generic embedding with critical point $\kappa$ such that $M\models$ $\varphi(\kappa,j(x))$. Then $\kappa\in j(\{\alpha<\kappa\st\varphi(\alpha,x)\}\cap C)$, and by elementarity $\{\alpha<\kappa\st\varphi(\alpha,x)\}\cap C\neq\emptyset$.
\end{proof}

The following corollary of Proposition \ref{proposition_formula} is the analogue of \cite[Corollary 2.2]{MR3913154}, which gives similar characterizations of various types of cardinals in terms of small embeddings.
\begin{corollary}[Folklore]
Let $\kappa$ be a cardinal.
\begin{enumerate}
\item $\kappa$ is uncountable and regular if and only if there is a generic embedding $j:V\to M\subseteq V[G]$ with critical point $\kappa$.
\item $\kappa$ is weakly inaccessible if and only if there is a generic embedding $j:V\to M\subseteq V[G]$ such that $M\models$ ``$\kappa$ is a cardinal.''
\item $\kappa$ is inaccessible if and only if there is a generic embedding $j:V\to M\subseteq V[G]$ with critical point $\kappa$ such that $M\models$ ``$\kappa$ is a strong limit.''
\item $\kappa$ is weakly Mahlo if and only if there is a generic embedding $j:V\to M\subseteq V[G]$ with critical point $\kappa$ such that $M\models$ ``$\kappa$ is regular.''
\item $\kappa$ is Mahlo if and only if there is a generic embedding $j:V\to M\subseteq V[G]$ with critical point $\kappa$ such that $M\models$ ``$\kappa$ is inaccessible.''
\end{enumerate}
\end{corollary}

\section{Large cardinal ideals and elementary embeddings}\label{section_embeddings}

\subsection{The $\Pi^m_n$-indescribability ideals}\label{section_indescribability_embeddings}


Recall that a formula is $\Pi^m_n$ if it starts with a block of universal quantifiers of type $m+1$ variables, followed by a block of type $m+1$ existential quantifiers, and so on with at most $n$ blocks in all, followed by a formula containing unquantified variables of type at most $m+1$ and quantified variables of type at most $m$. Similarly, a formula is $\Sigma^m_n$ if it begins with a block of type $m+1$ existential quantifiers. See \cite[Chapter 0]{MR1994835} for more details. 

A set $S\subseteq\kappa$ is \emph{$\Pi^m_n$-indescribable} if for every  $A\in V_{\kappa+1}$ and every $\Pi^m_n$-sentence $\varphi$ over $(V_\kappa,\in,A)$, whenever $(V_\kappa,\in,A)\models\varphi$ there is an $\alpha\in S$ such that $(V_\alpha,\in,A\cap V_\alpha)\models\varphi$. Levy \cite{MR0281606} showed that if $\kappa$ is a $\Pi^m_n$-indescribable cardinal then the collection
\[\Pi^m_n(\kappa)=\{X\subseteq\kappa\st\text{$X$ is not $\Pi^m_n$-indescribable}\}\]
is a normal ideal on $\kappa$, which is referred to as the \emph{$\Pi^m_n$-indescribable ideal} on $\kappa$ (see \cite[Proposition 6.11]{Kanamori:Book}). A set $S\subseteq\kappa$ is called \emph{weakly $\Pi^m_n$-indescribable} if for every $A\subseteq\kappa$ and every $\Pi^m_n$-sentence $\varphi$ over $(\kappa,\in,A)$, whenever $(\kappa,\in,A)\models\varphi$ there is an $\alpha\in S$ such that $(\alpha,\in,A\cap\alpha)\models\varphi$. It is easy to check \cite{MR0384553} that a set $S\subseteq\kappa$ is $\Pi^m_n$-indescribable if and only if $\kappa$ is inaccessible and $S$ is weakly $\Pi^m_n$-indescribable. Furthermore, when $\kappa$ is weakly $\Pi^m_n$-indescribable, the collection 
\[\widetilde{\Pi}^m_n(\kappa)=\{X\subseteq\kappa\st \text{$X$ is not weakly $\Pi^m_n$-indescribable}\}\]
of non-weakly $\Pi^m_n$-indescribable subsets of $\kappa$ is a nontrivial normal ideal on $\kappa$. 

\begin{remark}\label{remark_transfinite_indescribability}
Let us note here that Sharpe and Welch \cite{MR2817562} defined a notion of $\Pi^1_\xi$-indescribability of a cardinal $\kappa$ where $\xi<\kappa^+$ by demanding that the existence of a winning strategy for a particular player in a certain finite game played at $\kappa$ implies that the same player has a winning strategy for the game played at $\alpha$. Independently, Bagaria \cite{MR3894041} defined a natural notion of $\Pi^1_\xi$-formula for $\xi\geq\omega$. For example, a formula $\varphi$ is $\Pi^1_\omega$ if it is of the form $\bigwedge_{n<\omega}\varphi_n$ where each $\varphi_n$ is $\Pi^1_n$, and it contains only finitely-many free second-order variables. A set $S\subseteq\kappa$ is $\Pi^1_\xi$-indescribable if and only if for all $A\subseteq V_\kappa$ whenever $\varphi$ is $\Pi^1_\xi$ and $(V_\kappa,\in,A)\models\varphi$, there must be some $\alpha\in S$ such that $(V_\alpha,\in,A\cap V_\alpha)\models\varphi$. Independently, Brickhill-Welch \cite{BrickhillWelch} and Bagaria \cite{MR3894041} showed that if $\kappa$ is $\Pi^1_\xi$-indescribable where $\xi<\kappa$ then the collection
\[\Pi^1_\xi(\kappa)=\{X\subseteq\kappa\st\text{$X$ is not $\Pi^1_\xi$-indescribable}\}\]
is a normal ideal on $\kappa$.
\end{remark}

\begin{remark}
Recall that $\kappa$ is $\Pi^1_0$-indescribable if and only if $\kappa$ is inaccessible, and in this case the $\Pi^1_0$-indescribable ideal equals $\NS_\kappa$. 
\end{remark}

\begin{definition}
For notational convenience later on we let
\[\Pi^1_{-1}(\kappa)=[\kappa]^{<\kappa}\]
for all cardinals $\kappa$.
\end{definition}

\begin{lemma}\label{lemma_indescribable_filter}
If $\kappa$ is $\Pi^m_n$-indescribable, $A\in V_{\kappa+1}$ and $\varphi$ is a $\Pi^m_n$-sentence over $(V_\kappa,\in, A)$ such that $(V_\kappa,\in,A)\models\varphi$ then the set 
\[\{\alpha<\kappa\st (V_\alpha,\in,A\cap V_\alpha)\models\varphi\}\]
is in the $\Pi^m_n$-indescribable filter $\Pi^m_n(\kappa)^*$.
\end{lemma}

\begin{proof}
Suppose $(V_\kappa,\in,A)\models\varphi$ where $\varphi$ is a $\Pi^m_n$-sentence. Let $Z=\{\alpha<\kappa\st(V_\alpha,\in,A\cap V_\alpha)\models\varphi\}$. The set 
\[\kappa\setminus Z=\{\alpha<\kappa\st(V_\alpha,\in,A\cap V_\alpha)\models\lnot\varphi\}\]
is not $\Pi^m_n$-indescribable because if it were then the fact that $(V_\kappa,\in,A)\models\varphi$ would imply that $(V_\alpha,\in,A\cap V_\alpha)\models\varphi$ for some $\alpha\in\kappa\setminus Z$, a contradiction. Thus $\kappa\setminus Z\in\Pi^m_n(\kappa)$ which implies $Z\in \Pi^m_n(\kappa)^*$.
\end{proof}

The following is implicit in the work of Holy, L\"ucke and Njegomir \cite{MR3913154}, the only difference is that we generalize their characterization of $\Pi^m_n$-indescribable cardinals to subsets of $\kappa$ and we also give a generic embedding characterization.

\begin{proposition}[{Holy, L\"ucke and Njegomir \cite[Lemma 4.2]{MR3913154}}]\label{proposition_indescribability_embedding}
For $n,m<\omega$, $\kappa$ a regular cardinal and $S\subseteq\kappa$ the following are equivalent.
\begin{enumerate}
\item $S$ is $\Pi^m_n$-indescribable.
\item There is a generic embedding $j:V\to M\subseteq V[G]$ with critical point $\kappa$ such that $\kappa\in j(S)$ and for every $\Pi^m_n$-sentence $\varphi$ over $(V_\kappa,\in,A)$ where $A\in (V_{\kappa+1})^V$ we have\footnote{Note that $A\in M$ since $A=j(A)\cap V_\kappa$ and $j(A),V_\kappa\in M$.}
\[((V_\kappa,\in,A)\models\varphi)^V \implies ((V_\kappa, \in,A)\models\varphi)^M.\]
\item For all sufficiently large cardinals $\theta$ there is a small embedding $j:M\to H(\theta)$ for $\kappa$ such that
\begin{enumerate}
\item $S\in\ran(j)$,
\item $\crit(j)\in S$ and
\item for every $\Pi^m_n$-sentence $\varphi$ over $(V_{\crit(j)},\in,A)$ where $A\in M\cap V_{\crit(j)+1}$ we have
\[((V_{\crit(j)},\in,A)\models\varphi)^M \implies ((V_{\crit(j)},\in,A)\models\varphi)^V.\]
\end{enumerate}
\end{enumerate}
\end{proposition}

\begin{proof}
The proof that (1) and (3) are equivalent is very similar to \cite[Lemma 4.2]{MR3913154}. We leave the details to the reader.

To see that (1) implies (2), suppose $S\subseteq\kappa$ is $\Pi^m_n$-indescribable and let $G$ be generic for $P(\kappa)/(\Pi^m_n(\kappa)\restrict S)$. Let $U_G$ be the $V$-normal $V$-ultrafilter on $P(\kappa)^V$ obtained from $G$ and note that $U_G$ extends $(\Pi^m_n(\kappa)\restrict S)^*$, hence $S\in U_G$. Let $j:V\to M=\Ult(V,U_G)\subseteq V[G]$ be the generic ultrapower by $U_G$. Clearly, $\crit(j)=\kappa$ and $\kappa \in j(S)$. Let $\varphi(A)$ be a $\Pi^m_n$-sentence with parameter $A\in (V_{\kappa+1})^V$ such that $((V_\kappa,\in,A)\models\varphi)^V$. Since $\kappa$ is $\Pi^m_n$-indescribable in $V$, it follows from Lemma \ref{lemma_indescribable_filter} that the set $C=\{\alpha<\kappa\st(V_\alpha,\in,A\cap V_\alpha)\models\varphi\}$ is in the filter $(\Pi^m_n(\kappa)\restrict S)^*$ and is hence also in $U_G$. Thus, \[\kappa \in j(C)=\{\alpha \in j(\kappa)\st ((V_\alpha,E,j(A)\cap V_\alpha)\models\varphi)^M\},\] which implies $((V_\kappa,\in,A)\models\varphi)^M$.

For (2) implies (1), fix $S\subseteq\kappa$ and suppose that $j:V\to M\subseteq V[G]$ is a generic embedding as in (2). We will show that $S$ is $\Pi^m_n$-indescribable. Fix $A\in (V_{\kappa+1})^V$ and let $\varphi(A)$ be a $\Pi^m_n$-sentence over $(V_\kappa,\in,A)$ such that $(V_\kappa,\in,A)\models\varphi$. By the assumed properties of $j$ we obtain $((V_\kappa,\in,j(A)\cap V_\kappa)\models\varphi(j(A)\cap V_\kappa))^M$ and hence by elementarity, $V\models$ $(\exists\alpha\in S) (V_\alpha,\in,A\cap V_\alpha)\models\varphi(A\cap V_\alpha)$. Thus, $S$ is $\Pi^m_n$-indescribable. 
\end{proof}

\begin{remark}
Notice that in Proposition \ref{proposition_indescribability_embedding}, the existence of a single generic embedding suffices to characterize the $\Pi^m_n$-indescribability of $S\subseteq\kappa$, whereas, it may seem, we need to demand the existence of many small embeddings to characterize the $\Pi^m_n$-indescribability of $S$. However, the proof \cite[Lemma 4.2]{MR3913154} shows that, in fact, a single small embedding also suffices. That is, we can replace Proposition \ref{proposition_indescribability_embedding}(3) with the following.
\begin{itemize}
\item[(3)] There is a small embedding $j:M\to H(\beth_m(\kappa)^+)$ for $\kappa$ such that
\begin{itemize}
\item[(a)] $S\in \ran(j)$,
\item[(b)] $\crit(j)\in S$ and
\item[(c)] for every $\Pi^m_n$-sentence $\varphi$ over $(V_{\crit(j)},\in,A)$ where $A\in M\cap V_{\crit(j)+1}$ we have
\[((V_{\crit(j)},\in,A)\models\varphi)^M\implies((V_{\crit(j)},\in,A)\models\varphi)^V.\]
\end{itemize}
\end{itemize}
\end{remark}

\begin{remark}
It is not too difficult to see that the characterization given in Proposition \ref{proposition_indescribability_embedding}(2) can be extended to $\Pi^1_\xi$-indescribable subsets of $\kappa$ for all $\xi<\kappa^+$. See Remark \ref{remark_transfinite_indescribability} for the definition of $\Pi^1_\xi$-indescribability and \cite[Proposition 8.3]{CodyRefinement} for a proof of this result.
\end{remark}

\subsection{The subtle ideal}

Recall that for $S\subseteq\kappa$ we say that a sequence $\vec{S}=\<S_\alpha\st\alpha\in S\>$ is an \emph{$S$-list} if $S_\alpha\subseteq\alpha$ for all $\alpha\in S$. Jensen and Kunen defined a set $S\subseteq\kappa$ to be \emph{subtle} if for every $S$-list $\vec{S}=\<S_\alpha\st\alpha\in S\>$ and for every club $C\subseteq\kappa$ there exist $\alpha,\beta\in S\cap C$ with $\alpha<\beta$ such that $S_\alpha=S_\beta\cap \alpha$. Although the least subtle cardinal is not weakly compact, every subtle cardinal is a stationary limit (and more) of cardinals which are $\Pi^1_n$-indescribable for every $n<\omega$ (this follows from Lemma \ref{lemma_subtlety_is_stronger_than_indescribability} below).

To show that the collection of non-subtle subsets of a subtle cardinal forms a normal ideal let us fix some notation. We let $\gp:\ORD\times\ORD\to\ORD$ denote the standard definable pairing function. If $\alpha$ is a closure point of $\gp$, meaning $\gp[\alpha\times\alpha]\subseteq\alpha$, and if $\vec{A}=\<A_\xi\st\xi<\alpha\>$ is a sequence of subsets of $\alpha$, then we can use $\gp$ to code the sequence into a single subset of $\alpha$: $A=[[\vec{A}]]=\{\gp(\eta,\xi)\st \eta\in A_\xi\}$. The following lemma was used in \cite{MR0384553} and its proof is straightforward.

\begin{lemma}\label{lemma_godel_pairing}
Suppose $\alpha,\beta\in\ORD$ are closure points of $\gp$ with $\alpha<\beta$, and further suppose that $A\subseteq\alpha$ codes the sequence $\vec{A}=\<A_\xi\st\xi<\alpha\>$ of subsets of $\alpha$ and $B\subseteq\beta$ codes the sequence $\vec{B}=\<B_\xi\st\xi<\beta\>$ of subsets of $\beta$. In other words, $A=[[\vec{A}]]$ and $B=[[\vec{B}]]$. Then it follows that if $A=B\cap\alpha$ then $A_\xi=B_\xi\cap\alpha$ for all $\xi<\alpha$.
\end{lemma}

\begin{proposition}[Baumgartner \cite{MR0384553}]
If $\kappa$ is subtle then the collection 
\[\NSub_\kappa=\{X\subseteq\kappa\st\text{$X$ is not subtle}\}\]
is a normal ideal.
\end{proposition}

\begin{proof}
Let $S\subseteq\kappa$ be subtle and suppose $f:S\to\kappa$ is regressive. We must show that $f$ is constant on a subtle subset of $S$, that is, we must show that for some $\eta<\kappa$ we have $f^{-1}(\eta)\in\NSub_\kappa^+$. Suppose not. Let \[D=\{\alpha<\kappa\st\text{$\alpha$ is closed under the pairing function $\gp$}\}.\] Then $D$ is club and hence $S\cap D$ is subtle. By assumption, for each $\eta<\kappa$, we may let $\vec{S}^\eta=\<S^\eta_\alpha\st\alpha\in f^{-1}(\eta)\>$ be an $f^{-1}(\eta)$-list and let $C_\eta$ be a club such that for all $\alpha,\beta\in f^{-1}(\eta)\cap C_\eta$ with $\alpha<\beta$ we have $S^\eta_\alpha\neq S^\eta_\beta\cap\alpha$.

Define an $S\cap D$-list $\vec{S}=\<S_\alpha\st\alpha\in S\cap D\>$ by letting $S_\alpha=[[\<S^{f(\alpha)}_\alpha,\{f(\alpha)\}\>]]$. Let $C=\bigtriangleup_{\eta<\kappa}C_\eta$. Since $S\cap D$ is subtle there exists $\alpha,\beta\in S\cap D\cap C$ with $\alpha<\beta$ and $S_\alpha=S_\beta\cap\alpha$. Thus $[[\<S^{f(\alpha)}_\alpha,\{f(\alpha)\}\>]]=[[\<S^{f(\beta)}_\beta,\{f(\beta)\}\>]]\cap\alpha$ and since $\alpha$ and $\beta$ are both closed under G\"odel-pairing, it follows from Lemma \ref{lemma_godel_pairing} that $S^{f(\alpha)}_\alpha=S^{f(\beta)}_\beta\cap\alpha$ and $f(\alpha)=f(\beta)$. If we let $\eta=f(\alpha)=f(\beta)$ then we have $S^\eta_\alpha=S^\eta_\beta\cap\alpha$ where $\alpha,\beta\in f^{-1}(\eta)\cap C_\eta$, a contradiction.
\end{proof}

In the following result we give a characterization of subtle sets in terms of elementary embeddings. The fact that (1) and (3) are equivalent in the following is due to Holy, L\"ucke and Njegomir \cite[Lemma 5.2]{MR3913154}; we provide a proof here since it will be referenced below in order to emphasize a difference in character between generic embedding and small embedding characterizations of large cardinal ideals.

\begin{proposition}\label{proposition_subtle_embedding}
For all cardinals $\kappa$ and all $S\subseteq\kappa$, the following are equivalent.
\begin{enumerate}
\item $S\subseteq\kappa$ is subtle
\item There is a generic elementary embedding $j:V\to M\subseteq V[G]$ with critical point $\kappa$ and $\kappa\in j(S)$ such that for every $S$-list $\vec{S}=\<S_\alpha\st\alpha\in S\>$ and every club $C\subseteq\kappa$ in $V$, we have $S_\alpha=j(\vec{S})(\kappa)\cap\alpha$ for some $\alpha\in S\cap C$.
\item For every $S$-list $\vec{S}=\<S_\alpha\st\alpha\in S\>$ and for every club $C\subseteq\kappa$, there is a small embedding $j:M\to H(\kappa^+)$ for $\kappa$ such that $S,C,\vec{S}\in\ran(j)$, $\crit(j)\in S$ and $S_\alpha=S_{\crit(j)}\cap\alpha$ for some $\alpha\in C\cap \crit(j)$.
\end{enumerate}
\end{proposition}

\begin{proof}
For (1) implies (3), suppose $S$ is subtle, $\vec{S}=\<S_\alpha\st\alpha\in S\>$ is an $S$-list and $C\subseteq\kappa$ is a club. Let $\<X_\alpha\st\alpha<\kappa\>$ be a continuous increasing sequence of elementary substructures of $H(\kappa^+)$ of cardinality less than $\kappa$ with $\vec{S},C\in X_0$ and $\alpha\subseteq X_\alpha\cap\kappa\in\kappa$ for all $\alpha<\kappa$. Then $D=\{\alpha\in C\st\alpha= X_\alpha\cap\alpha\}$ is club in $\kappa$ and since $S\cap D$ is subtle there are $\alpha,\beta\in S\cap D$ with $\alpha<\beta$ and $S_\alpha=S_\beta\cap\alpha$. Let $\pi:X_\beta\to M$ be the transitive collapse of $X_\beta$. Then $j=\pi^{-1}:M\to H(\kappa^+)$ is a small embedding for $\kappa$ with $\crit(j)=\beta$, $\vec{S}$, $C\in\ran(j)$ and $S_\alpha=S_{\crit(j)}\cap\alpha$. The converse, namely (3) implies (1), is quite straight forward.

To see that (1) implies (2), suppose $S\subseteq\kappa$ is subtle. Clearly $S$ is stationary in $\kappa$. Let $G\subseteq P(\kappa)/(\NSub_\kappa\restrict S)$ be generic. Let $U_G$ be the $V$-normal $V$-ultrafilter obtained from $G$ and let $j:V\to V^\kappa/U_G\subseteq V[G]$ be the generic ultrapower by $U_G$. Then the critical point of $j$ is $\kappa$ and $\kappa\in j(S)$ since $S\in U_G$. Fix an $S$-list $\vec{S}=\<S_\alpha\st\alpha\in S\>$ and a club $C\subseteq\kappa$. Let us show that the set 
\[X=\{\xi\in S\st (\exists\zeta\in S\cap C\cap\xi)\ S_\zeta=S_\xi\cap\zeta\}\] is in $(\NSub_\kappa\restrict S)^*$; in other words, we will show that $\kappa\setminus X\in\NSub_\kappa\restrict S$, i.e. $(\kappa\setminus X)\cap S\in\NSub_\kappa$. Since 
\[(\kappa\setminus X)\cap S=\{\xi\in S\st (\forall\zeta\in  S\cap C\cap\xi)\ S_\zeta\neq S_\xi\cap\zeta\}\] we have that $S_\alpha\neq S_\beta\cap\alpha$ for all $\alpha,\beta\in C\cap (\kappa\setminus X)\cap S$ with $\alpha<\beta$. Hence $(\kappa\setminus X)\cap S$ is not subtle. Thus $X\in(\NSub_\kappa\restrict S)^*\subseteq U_G$ which implies that $\kappa\in j(X)$ and hence there is an $\alpha\in  S\cap C$ such that $S_\alpha=j(\vec{S})(\kappa)\cap\alpha$.

Conversely, for (2) implies (1), fix an $S$-list $\vec{S}=\<S_\alpha\st\alpha\in S\>$ and a club $C\subseteq\kappa$ and suppose $j:V\to M\subseteq V[G]$ is a generic elementary embedding with critical point $\kappa$ such that $\kappa\in j(S)$ and $S_\alpha=j(\vec{S})(\kappa)\cap\alpha$ for some $\alpha\in S\cap C$. Since $\alpha,\kappa\in j(S\cap C)$, it follows by elementarity that there exist $\zeta,\xi\in S\cap C$ with $\zeta<\xi$ such that $S_\zeta=S_\xi\cap\zeta$.
\end{proof}

\begin{remark}\label{remark_subtle_difference}
Let us point out that there is a substantial difference between the generic embedding characterization and the small embedding characterization of subtlety given in the previous proposition, which is not manifested in the embedding characterizations of stationarity. Namely, in the generic embedding characterization of the subtlety of $S\subseteq\kappa$, Proposition \ref{proposition_subtle_embedding}(2), a single embedding $j$ works for all $S$-lists and all clubs $C$. Whereas, in the small embedding characterization, Proposition \ref{proposition_subtle_embedding}(3), for each $S$-list there is a small embedding with the desired property. This seems to be an inherent difference between these two characterizations because one cannot place all club subsets of $\kappa$ into an elementary substructure $X_0\prec H(\theta)$ while still guaranteeing that $X_0\cap\kappa\in\kappa$. Indeed, it is easy to see that for regular cardinals $\omega<\kappa<\theta$, if $X\prec H(\theta)$ and $X\cap\kappa\in\kappa$ then $X$ does not contain all club subsets of $\kappa$. Suppose not. Let $X\prec H(\theta)$ contain all club subsets of $\kappa$ with $X\cap\kappa\in\kappa$. Then $\crit(j)=X\cap\kappa\in C$ for all clubs $C\subseteq\kappa$, which is a contradiction.
\end{remark}

As it will be needed below for a characterization of ineffability, let us present a result of Baumgartner, which shows that if $\kappa$ is a subtle cardinal then there are many cardinals less than $\kappa$ which are $\Pi^1_n$-indescribable for all $n<\omega$. Recall that for a cardinal $\kappa$, $S\subseteq\kappa$ and an $S$-list $\vec{S}=\<S_\alpha\st\alpha\in S\>$, a set $X\subseteq\kappa$ is \emph{homogeneous} for $\vec{S}$ if and only if whenever $\alpha,\beta\in X$ with $\alpha<\beta$ we have $S_\alpha=S_\beta\cap\alpha$.

\begin{lemma}[Baumgartner \cite{MR0384553}]\label{lemma_subtlety_is_stronger_than_indescribability}
Let $\kappa$ be a cardinal and $S\subseteq\kappa$. Assume $S$ is subtle and $\vec{S}=\<S_\alpha\st\alpha\in S\>$ is an $S$-list. Let 
\begin{align*}
A=\{\alpha\in S\st \exists X\subseteq S\cap\alpha & \text{ such that $X$ is $\Pi^1_n$-indescribable in $\alpha$}\\
	&\text{for all $n<\omega$ and $X$ is homogeneous for $\vec{S}$}\}
\end{align*}
Then $S\setminus A$ is not subtle.
\end{lemma}

\begin{proof}

Suppose $S\setminus A$ is a subtle subset of $\kappa$. Fix a bijection $b:V_\kappa\to \kappa$ and let $C$ be a club subset of $\kappa$ such that for all $\alpha\in C$ we have $b\restrict V_\alpha : V_\alpha \to \alpha$ is a bijection and $\alpha$ is closed under the standard definable pairing function $\gp:\kappa\times\kappa\to\kappa$. By normality of the subtle ideal, it follows that $E:=(S\setminus A)\cap C$ is a subtle subset of $\kappa$. We define an $E$-list $\vec{E}$ as follows. Fix $\beta\in E$ and let $B_\beta=\{\alpha\in S\cap\beta\st S_\alpha=S_\beta\cap\alpha\}$. Notice that $B_\beta\cup\{\beta\}$ is homogeneous for $\vec{S}$. Since $\beta\in E$, it follows that for some $n<\omega$ the set $B_\beta$ is not $\Pi^1_n$-indescribable in $\beta$. Let $\varphi_\beta$ be a $\Pi^1_n$-sentence and let $A_\beta\subseteq V_\beta$ be such that $(V_\beta,\in,\gp,A_\beta)\models\varphi_\beta$ while for all $\alpha\in B_\beta$ we have $(V_\alpha,\in,\gp\restrict\alpha,A_\beta\cap\alpha)\models\lnot\varphi_\beta$. Let $E_\beta=[[(S_\beta,\{\varphi_\beta\},b[A_\beta])]]$. This defines an $E$-list $\vec{E}=\<E_\beta\st\beta\in E\>$. 

Since $E$ is subtle there exist $\alpha<\beta$ in $E$ such that $E_\alpha=E_\beta\cap\alpha$. It follows that $S_\alpha=S_\beta\cap\alpha$ and thus $\alpha\in B_\beta$. Furthermore, $b[A_\alpha]=b[A_\beta]\cap \alpha$ and hence $A_\alpha= A_\beta\cap V_\alpha$. Also we have $\varphi_\alpha=\varphi_\beta=\varphi$. By definition of $\varphi_\alpha$ we see that $(V_\alpha,\in,\gp,A_\alpha)\models\varphi$, and since $\alpha\in B_\beta$ we have $(V_\alpha,\in,\gp\restrict\alpha,A_\beta\cap V_\alpha)\models\lnot\varphi$, a contradiction.
\end{proof}

An easy corollary of the previous lemma shows that, although the definition of subtlety only asserts the existence of homogeneous sets of size 2, one can obtain much larger homogeneous sets for free; quoting Abramson-Harrington-Kleinberg-Zwicker, ``a little coherence goes a long way'' (see \cite[Section 3.6]{MR0460120}).

\begin{corollary}[Baumgartner \cite{MR0384553}]\label{corollar_indescribable_hom_sets_from_subtlety}
A set $S\subseteq\kappa$ is subtle if and only if for every $S$-list $\vec{S}=\<S_\alpha\st\alpha\in S\>$ and every club $C\subseteq\kappa$ there is an $\alpha\in S\cap C$ such that there is a set $H\subseteq S\cap C\cap\alpha$ homogeneous for $\vec{S}$ which is $\Pi^1_n$-indescribable for all $n<\omega$.
\end{corollary}

As another corollary of Lemma \ref{lemma_subtlety_is_stronger_than_indescribability}, we easily obtain an additional characterization of subtlety in terms of generic elementary embeddings.

\begin{corollary}
For all cardinals $\kappa$ and $S\subseteq\kappa$ the following are equivalent.
\begin{enumerate}
\item $S$ is subtle.
\item There is a generic embedding $j:V\to M\subseteq V[G]$ with critical point $\kappa$ and $\kappa\in j(S)$ such that for every $S$-list $\vec{S}=\<S_\alpha\st\alpha\in S\>$ and every club $C\subseteq\kappa$ in $V$, we have $S_\alpha=j(\vec{S})(\kappa)\cap\alpha$ for some $\alpha\in S\cap C$ and furthermore, $M\models$ ``there is a set $H\subseteq S$ which is $\Pi^1_n$-indescribable in $\kappa$ for all $n<\omega$ and homogeneous for $\vec{S}$''.
\end{enumerate}
\end{corollary}

\begin{proof}
For (2) implies (1) the proof is the same as that of Proposition \ref{proposition_subtle_embedding}. For (1) implies (2), assuming $S$ is subtle, let $G\subseteq P(\kappa)/(\NSub_\kappa\restrict S)$ be generic, let $U_G$ be the corresponding $V$-normal ultrafilter and let $j:V\to V^\kappa/U_G\subseteq V[G]$ be the generic ultrapower. As in the proof of Proposition \ref{proposition_subtle_embedding}, one can check that there is an $\alpha\in S\cap C$ such that $S_\alpha=j(\vec{S})(\kappa)\cap\alpha$. For the remaining statement, let 
\begin{align*}
A=\{\alpha\in S\st \exists X\subseteq S\cap\alpha & \text{ such that $X$ is $\Pi^1_n$-indescribable in $\alpha$}\\
	&\text{for all $n<\omega$ and $X$ is homogeneous for $\vec{S}$}\}
\end{align*}
and notice that by Lemma \ref{lemma_subtlety_is_stronger_than_indescribability}, $S\cap(\kappa\setminus A)\in\NSub_\kappa$, which implies $A\in(\NSub_\kappa\restrict S)^*$ and hence $\kappa\in j(A)$.
\end{proof}

\subsection{The almost ineffable and ineffable ideals}


Given a regular cardinal $\kappa$ and an ideal $I\supseteq [\kappa]^{<\kappa}$ we define another ideal $\I(I)$ by letting $S\notin\I(I)$ if and only if for every $S$-list $\vec{S}=\<S_\alpha\st\alpha\in S\>$ there is a set $H\in P(S)\cap I^+$ which is homogeneous for $\vec{S}$. We say that a set $S\subseteq\kappa$ is \emph{almost ineffable} if $S\in \I([\kappa]^{<\kappa})^+$, that is, every $S$-list $\vec{S}=\<S_\alpha\st\alpha\in S\>$ has a homogeneous set $H\subseteq S$ of size $\kappa$. Similarly, $S$ is \emph{ineffable} if $S\in\I(\NS_\kappa)^+$, that is, every $S$-list $\vec{S}=\<S_\alpha\st\alpha\in S\>$ has a homogeneous set $H\subseteq S$ which is stationary in $\kappa$. Baumgartner \cite{MR0384553} showed that if $\kappa$ is ineffable then the ideal $\I(\NS_\kappa)$ is nontrivial and normal, and similarly, if $\kappa$ is almost ineffable then $\I([\kappa]^{<\kappa})$ is nontrivial and normal. The collection $\I([\kappa]^{<\kappa})$ is called the \emph{almost ineffable ideal} and $\I(\NS_\kappa)$ is called the \emph{ineffable ideal}. It is easy to see that for any ideal $I\supseteq[\kappa]^{<\kappa}$ we have $S\in\I(I)^+$ if and only if for every $S$-list $\vec{S}=\<S_\alpha\st\alpha\in S\>$ there is a set $D\subseteq\kappa$ which is anticipated by $S$ on a set in $I^+$, meaning that the set $\{\alpha\in S\st S_\alpha=D\cap\alpha\}$ is in $I^+$. This easily leads to the following characterization of the ineffability of a set $S\subseteq\kappa$; note that below we give another characterization of ineffability (see Proposition \ref{proposition_ineffable_embedding_baumgartner}).

\begin{proposition}\label{proposition_ineffable_embedding_stationary}
For all cardinals $\kappa$ and $S\subseteq\kappa$ the following are equivalent.
\begin{enumerate}
\item $S$ is ineffable.
\item For all $S$-lists $\vec{S}=\<S_\alpha\st\alpha\in S\>$ there is a generic elementary embedding $j:V\to M\subseteq M[G]$ with critical point $\kappa$ and $\kappa\in j(S)$ such that $j(\vec{S})(\kappa)\in V$.
\item For all $S$-lists $\vec{S}=\<S_\alpha\st\alpha\in S\>$ there is a small embedding $j:M\to H(\kappa^+)$ for $\kappa$ with $S,\vec{S}\in\ran(j)$ and $\crit(j)\in S$ such that $S_{\crit(j)}\in M$.
\end{enumerate}
\end{proposition}

\begin{proof}
The proof that (1) and (3) are equivalent is very similar to that of \cite[Lemma 5.5]{MR3913154}. For (1) implies (2), suppose $S$ is ineffable, let $\vec{S}=\<S_\alpha\st\alpha\in S\>$ be an $S$-list and let $D\subseteq\kappa$ be such that the set $A=\{\alpha\in S\st S_\alpha=D\cap\alpha\}$ is stationary. Let $G\subseteq P(\kappa)/(\NS_\kappa\restrict A)$ and let $j:V\to M=V^\kappa/U_G\subseteq V[G]$ be the corresponding generic ultrapower. Then, $\crit(j)=\kappa$ and $\kappa\in j(S)$. Furthermore, since $A\in U_G$ it follows that $\kappa\in j(A)$ and hence $j(\vec{S})(\kappa)=D\in V$. 

For (2) implies (1), fix an $S$-list $\vec{S}=\<S_\alpha\st\alpha\in S\>$. We must show that $\vec{S}$ has a homogeneous set that is stationary in $\kappa$. Using (2), there is a generic elementary embedding $j:V\to M\subseteq V[G]$ with critical point $\kappa$ and $\kappa\in j(S)$ such that $D=j(\vec{S})(\kappa)\in V$. Let $H=\{\alpha\in S\st S_\alpha=D\cap\alpha\}$. Clearly $H\in V$ and $H$ is homogeneous for $\vec{S}$. To see that $H$ is stationary in $\kappa$, fix a club $C\subseteq\kappa$ and notice that $\kappa\in j(H)\cap j(C)$, and hence $H\cap C\neq\emptyset$ by elementarity.
\end{proof}

By combining the arguments of Proposition \ref{proposition_indescribability_embedding} and Proposition \ref{proposition_ineffable_embedding_stationary}, we obtain the following characterization of the ideals $\I(\Pi^1_n(\kappa))$ for $n<\omega$, which generalizes the generic embedding characterization of ineffability from Proposition \ref{proposition_ineffable_embedding_stationary}; it is not clear whether or not the small embedding characterization from Proposition \ref{proposition_ineffable_embedding_stationary} can also be generalized in this way.

\begin{proposition}\label{proposition_ineffability_anticipation_embedding}
For all $n<\omega$, all cardinals $\kappa$ and all $S\subseteq\kappa$, the following are equivalent.
\begin{enumerate}
\item $S\in\I(\Pi^1_n(\kappa))^+$
\item For all $S$-lists $\vec{S}=\<S_\alpha\st\alpha\in S\>$ there is a generic elementary embedding $j:V\to M\subseteq V[G]$ with critical point $\kappa$ and $\kappa\in j(S)$ such that the following properties hold.
\begin{enumerate}
\item $j(\vec{S})(\kappa)\in V$
\item For all $R\in V_{\kappa+1}^V$ and all $\Pi^1_n$-sentences $\varphi$ we have
\[((V_\kappa,\in,R)\models\varphi)^V\implies((V_\kappa,\in,R)\models\varphi)^M.\]
\end{enumerate}
\end{enumerate}
\end{proposition}

\begin{proof}
For (1) implies (2), suppose $S\in\I(\Pi^1_n(\kappa))^+$ and fix an $S$-list $\vec{S}=\<S_\alpha\st\alpha\in S\>$. Let $D\subseteq\kappa$ be such that the set $A=\{\alpha\in S\st S_\alpha=D\cap\alpha\}$ is $\Pi^1_n$-indescribable. Let $G\subseteq P(\kappa)/(\Pi^1_n(\kappa)\restrict A)$ be generic and let $j:V\to M=V^\kappa/U_G\subseteq V[G]$ be the corresponding generic ultrapower embedding. Since $\Pi^1_n(\kappa)\restrict S$ is a normal ideal on $\kappa$, it follows that $U_G$ is a $V$-normal ultrafilter. Hence $\crit(j)=\kappa$ and $M$ is well-founded up to $\kappa^+$. Since $A\in(\Pi^1_n(\kappa) \restrict A)^*\subseteq U_G$, it follows that $\kappa\in j(A)$ and thus $j(\vec{S})(\kappa)=D$, so (2)(a) holds. Since $\Pi^1_n(\kappa)^*\subseteq(\Pi^1_n(\kappa)\restrict A)^*\subseteq U_G$, the same argument as that of Proposition \ref{proposition_indescribability_embedding} shows that (2)(b) holds. 

Conversely, suppose (2) holds. Fix an $S$-list $\vec{S}=\<S_\alpha\st\alpha\in S\>$ and let $j$ be a generic elementary embedding as in (2). Let $D=j(\vec{S})(\kappa)$. We will argue that $A=\{\alpha\in S\st S_\alpha=D\cap\alpha\}$ is $\Pi^1_n$-indescribable, which is sufficient since $A$ is clearly homogeneous for $\vec{S}$. By (2)(a) it follows that $D\in V$, and thus $A\in V$. Suppose $R\in V_{\kappa+1}^V$ and $\varphi$ is a $\Pi^1_n$-sentence such that $((V_\kappa,\in,R)\models\varphi)^V$. By (2)(b), it follows that $((V_\kappa,\in,R)\models\varphi)^M$. Since $j(\vec{S})(\kappa)=j(D)\cap\kappa$, we see that $\kappa\in j(A)$, and thus, it follows by elementarity that there is some $\alpha\in A$ such that $((V_\alpha,\in,R\cap V_\alpha)\models\varphi)^V$. In other words, $A$ is $\Pi^1_n$-indescribable.
\end{proof}

\begin{remark}
Notice that the case $n=-1$, where $\I(\Pi^1_{-1}(\kappa))=\I([\kappa]^{<\kappa})$ is conspicuously missing from Proposition \ref{proposition_ineffable_embedding_stationary} and Proposition \ref{proposition_ineffability_anticipation_embedding}. One reason for this is that the assumption that $S\subseteq\kappa$ is almost ineffable gives for each $S$-list $\vec{S}=\<S_\alpha\st\alpha\in S\>$ a set $D\subseteq\kappa$ such that $A=\{\alpha\in S\st S_\alpha=D\cap\alpha\}$ has size $\kappa$, and forcing with the non-normal ideal $[\kappa]^{<\kappa}\restrict A$ is problematic. Nonetheless, we give a generic embedding characterization of almost ineffable sets below in Proposition \ref{proposition_ineffable_embedding_baumgartner}.
\end{remark}

Baumgartner proved \cite{MR0384553} that, for a cardinal $\kappa$ and $S\subseteq\kappa$, demanding every $S$-list has a $\Pi^1_n$-indescribable homogeneous set implies $S$ is $\Pi^1_{n+1}$-indescribable. We include a proof here for the reader's convenience. Together with Theorem \ref{theorem_baumgartner_ineffable_ideal}, this result will be use below in the proof of Proposition \ref{proposition_ineffable_embedding_baumgartner} to establish another generic embedding characterization of the ideal $\I(\Pi^1_n(\kappa))$.

\begin{lemma}[Baumgartner \cite{MR0384553}]\label{lemma_ineffability_implies_indescribability}
Suppose $\kappa$ is a cardinal and $S\subseteq\kappa$ is such that every $S$-list $\vec{S}=\<S_\alpha\st\alpha\in S\>$ has a homogeneous set which is $\Pi^1_n$-indescribable where $n\in\{-1\}\cup\omega$. Then $S$ is $\Pi^1_{n+2}$-indescribable.
\end{lemma}

\begin{proof}
Let us first handle the case in which $n=-1$. Suppose every $S$-list has a homogeneous set $H$ in $(\Pi^1_{-1}(\kappa))^+=([\kappa]^{<\kappa})^+$. In other words, $S$ is almost ineffable. Let us show that $S$ is $\Pi^1_1$-indescribable. Since the almost ineffability of $S$ implies that $\kappa$ is inaccessible, it suffices to show that $S$ is weakly $\Pi^1_1$-indescribable (although showing that $S$ is $\Pi^1_n$-indescribable in $\kappa$ is not much harder using a bijection $b:V_\kappa\to\kappa$ and a club $C\subseteq\kappa$ of closure points of $b$). Suppose not. Then there is an $A\subseteq\kappa$ and a $\Pi^1_1$-sentence $\forall X\psi(X)$ such that $(\kappa,\in,A)\models\forall X\psi(X)$ and for all $\alpha\in S$ there is an $S_\alpha\subseteq \alpha$ such that $(\alpha,\in,A\cap\alpha,S_\alpha)\models\lnot\psi(S_\alpha)$. Let $H\subseteq S$ be homogeneous for $\vec{S}=\<S_\alpha\st\alpha\in S\>$ with $|H|=\kappa$ and let $X=\bigcup_{\alpha\in S}S_\alpha$. Then $(\kappa,\in,A,X)\models \psi(X)$ where $\psi(X)$ is a $\Pi^1_0$ sentence. Since $\kappa$ is inaccessible and $S$ is stationary, it follows that $S$ is $\Pi^1_0$-indescribable. Thus, there is an $\alpha\in S$ such that $(\alpha,\in,A\cap\alpha,X\cap\alpha)\models\psi(X\cap\alpha)$, but since $X\cap\alpha=S_\alpha$, this is a contradiction.


Now suppose $n\geq 0$ and every $S$-list $\vec{S}$ has a $\Pi^1_n$-indescribable homogeneous set $H$. Then $\kappa$ is inaccessible, so to show that $S$ is $\Pi^1_{n+2}$-indescribable, it suffices to show that $S$ is weakly $\Pi^1_{n+2}$-indescribable. Suppose $S$ is not weakly $\Pi^1_{n+2}$-indscribable. Then there is an $A\subseteq\kappa$ and a $\Pi^1_{n+2}$ sentence $\forall X\exists Y\psi(X,Y)$, where $\psi(X,Y)$ is a $\Pi^1_n$ formula, such that $(\kappa,\in,A)\models\forall X\exists Y\psi(X,Y)$ and for all $\alpha\in S$ there is an $S_\alpha\subseteq \alpha$ such that $(\alpha,\in,A\cap\alpha,S_\alpha)\models\forall Y\lnot\psi(S_\alpha,Y)$. Then $\vec{S}=\<S_\alpha\st\alpha\in S\>$ is an $S$-list and we may let $H\subseteq S$ be a $\Pi^1_n$-indescribable homogeneous set for $\vec{S}$. Since $(\kappa,\in,A)\models\forall X\exists Y\psi(X,Y)$, if we let $X=\bigcup_{\alpha\in H}S_\alpha$ then there is a $Y\subseteq\kappa$ such that $(\kappa,\in,A,X,Y)\models\psi(X,Y)$. Since $S$ is weakly $\Pi^1_n$-indescribable, and the set 
\[C=\{\alpha<\kappa\st(\alpha,\in,A\cap\alpha,X\cap\alpha,Y\cap\alpha)\models\psi(X,Y)\}\]
is in the weakly $\Pi^1_n$-indescribable filter, it follows that $S\cap C$ is weakly $\Pi^1_n$-indescrib\-able. Thus, there is some $\alpha\in S\cap C$ such that 
\[(\alpha,\in,A\cap\alpha,X\cap\alpha,Y\cap\alpha)\models\psi(X\cap\alpha,Y\cap\alpha).\]
Since $X\cap\alpha=S_\alpha$ (by homogeneity of $H$), it follows that
\[(\alpha,\in,A\cap\alpha,S_\alpha)\models\exists Y\psi(S_\alpha,Y),\]
a contradiction.
\end{proof}

\begin{remark}\label{remark_baumgartners_characterizations_of_ineffability}
Baumgartner showed that a cardinal $\kappa$ is almost ineffable if and only if it is subtle, $\Pi^1_1$-indescribable, and the ideal generated by the subtle ideal together with the $\Pi^1_1$-indescribable ideal is nontrivial and equals the almost ineffable ideal; moreover, the statement about the ideals cannot be removed from this characterization since the least cardinal which is subtle and $\Pi^1_1$-indescribable is not almost ineffable (this follows from Theorem \ref{theorem_baumgartner_ineffable_ideal}). Similarly, $\kappa$ is ineffable if and only if it is subtle, $\Pi^1_2$-indescribable and the ideal generated by the subtle ideal and the $\Pi^1_2$-indescribable ideal is nontrivial and equals the ineffable ideal; as above, the statement about ideals cannot be removed from the characterization because the least subtle cardinal which is $\Pi^1_2$-indescribable is not ineffable (this follows from Theorem \ref{theorem_baumgartner_ineffable_ideal}). In fact, Baumgartner proved a more general result, which appears as Theorem \ref{theorem_baumgartner_ineffable_ideal} below.
\end{remark}

Let us consider the ideals $\I(\Pi^1_n(\kappa))$ where $n\in\{-1\}\cup\omega$. Recall, using the notation introduced above, $\I(\Pi^1_{-1}(\kappa))=\I([\kappa]^{<\kappa})$ is the almost ineffable ideal and $\I(\Pi^1_0(\kappa))=\I(\NS_\kappa)$ is the ineffable ideal. By definition, a set $S\subseteq\kappa$ is in $\I(\Pi^1_n(\kappa))^+$ if and only if every $S$-list has a homogeneous set $H\subseteq S$ which is $\Pi^1_n$-indescribable in $\kappa$. Since ineffability implies almost ineffability, which in turn implies subtlety, it is easy to see that
\[\NSub_\kappa\subseteq\I([\kappa]^{<\kappa})\subseteq\I(\NS_\kappa)\subseteq\I(\Pi^1_1(\kappa))\subseteq\I(\Pi^1_2(\kappa))\subseteq\cdots.\]
In fact, by applying Lemma \ref{lemma_ineffability_implies_indescribability}, one can show that the above containments are proper when the ideals are nontrivial. Let us show that $\NSub_\kappa\subsetneq\I([\kappa]^{<\kappa})$ when $\kappa$ is almost ineffable (see Corollary \ref{corollary_sub_prop_contained} below). Similar arguments, which are left to the reader, establish the remaining proper containments (for a more general result see \cite[Theorem 1.5]{Holy-Lucke}).

Recall that if $S$ is stationary in $\kappa$ and for each $\alpha\in S$ we have a set $S_\alpha$ which is stationary in $\alpha$, then $\bigcup_{\alpha\in S}S_\alpha$ is stationary in $\kappa$. The analog of this result also holds for large cardinal ideals. For example, the stationary limit of subtle cardinals is subtle by the following Lemma.

\begin{lemma}\label{lemma_stationry_limit}
Suppose $S\in\NS_\kappa^+$ and that for each $\alpha\in S$ we have a set $S_\alpha\in\NSub_\alpha^+$. Then $\bigcup_{\alpha\in S}S_\alpha\in\NSub_\kappa^+$.
\end{lemma}

\begin{proof}
To show that $T=\bigcup_{\alpha\in S}S_\alpha$ is subtle in $\kappa$, fix a $T$-list $\vec{T}=\<T_\alpha\st\alpha\in T\>$ and a club $C\subseteq\kappa$. Since $S$ is stationary in $\kappa$, we may choose $\alpha\in S\cap C'$. Since $S_\alpha$ is subtle in $\alpha$ and $C\cap\alpha$ is a club subset of $\alpha$, there exist $\xi,\zeta\in S_\alpha\cap C\subseteq T\cap C$ with $\xi<\zeta$ such that $T_\xi= T_\zeta\cap \xi$. Thus, $T$ is subtle in $\kappa$.
\end{proof}

\begin{lemma}\label{lemma_set_of_non_subtles}
If $\kappa$ is subtle then the set
\[T=\{\alpha<\kappa\st\text{$\alpha$ is not subtle}\}\]
is subtle.
\end{lemma}

\begin{proof}
For the sake of contradiction, suppose $\kappa$ is the least counterexample. That is, $\kappa$ is the least cardinal such that $\kappa$ is subtle and $\{\alpha<\kappa\st\text{$\alpha$ is not subtle}\}$ is not subtle. Then $S:=\kappa\setminus T\in\NSub_\kappa^*$ and hence $S\in\NSub_\kappa^+$. For each $\alpha\in S$, by the minimality of $\kappa$, the set $S_\alpha:=T\cap\alpha$ is subtle in $\alpha$. Thus, by Lemma \ref{lemma_stationry_limit}, the set $T=\bigcup_{\alpha\in S}S_\alpha$ is subtle in $\kappa$, a contradiction.
\end{proof}

\begin{corollary}\label{corollary_sub_prop_contained}
Suppose $\kappa$ is almost ineffable. Then 
\[\NSub_\kappa\subsetneq \I([\kappa]^{<\kappa}).\]
\end{corollary}

\begin{proof}
Lemma \ref{lemma_set_of_non_subtles} implies that $T=\{\alpha<\kappa\st\text{$\alpha$ is not subtle}\}\notin\NSub_\kappa$. Furthermore, by definition of subtlety, there is a $\Pi^1_1$ sentence $\varphi$ such that for all $\alpha\leq\kappa$ we have
\[\text{$\alpha$ is subtle if and only if $V_\alpha\models\varphi$}.\]
Since $\kappa$ is almost ineffable, it is also subtle, and hence $V_\kappa\models\varphi$. By Lemma \ref{lemma_ineffability_implies_indescribability}, $\kappa$ is $\Pi^1_1$-indescribable and hence the set
\[\kappa\setminus T=\{\alpha<\kappa\st\text{$\alpha$ is subtle}\}=\{\alpha<\kappa\st V_\alpha\models\varphi\}\]
is in $\Pi^1_1(\kappa)^*$. By Lemma \ref{lemma_ineffability_implies_indescribability}, it follows that $\I([\kappa]^{<\kappa})^+\subseteq\Pi^1_1(\kappa)^+$ and hence $\Pi^1_1(\kappa)^*\subseteq\I([\kappa]^{<\kappa})^*$. Therefore, we see that $\kappa\setminus T\in\I([\kappa]^{<\kappa})^*$ and now it easily follows that $T\in I([\kappa]^{<\kappa})\setminus\NSub_\kappa$.
\end{proof}

Notice that taking $n=-1$ in the following yields Baumgartner's characterization of almost ineffability, and taking $n=0$ yields his characterization of ineffability, both of which were mentioned above in Remark \ref{remark_baumgartners_characterizations_of_ineffability}.

\begin{theorem}[Baumgartner \cite{MR0384553}]\label{theorem_baumgartner_ineffable_ideal}
For all $n\in\{-1\}\cup\omega$, we have $\kappa\notin\I(\Pi^1_n(\kappa))$ if and only if the following properties hold.
\begin{enumerate}
\item $\kappa$ is subtle and $\Pi^1_{n+2}$-indescribable.
\item The ideal generated by the subtle ideal and the $\Pi^1_{n+2}$-indescribable ideal is a nontrivial normal ideal; in this case, \[\I(\Pi^1_n(\kappa))=\overline{\NSub_\kappa\cup\Pi^1_{n+2}(\kappa)}.\]
\end{enumerate}
Furthermore, (2) cannot be removed from this characterization because the least cardinal $\kappa$ which is subtle and $\Pi^1_{n+2}$-indescribable is in $\I(\Pi^1_n(\kappa))$.
\end{theorem}

\begin{proof}
Let $I=\overline{\NSub_\kappa\cup\Pi^1_{n+2}(\kappa)}$. We will show that $S\in\I(\Pi^1_n(\kappa))^+$ if and only if $S\in I^+$.

Suppose $S\in \I(\Pi^1_n(\kappa))^+$. To show $S\in I^+$, it suffices to show $S$ is subtle and $\Pi^1_{n+2}$-indescribable by Remark \ref{remark_ideal_containment}. Clearly, $S$ is subtle, and by Lemma \ref{lemma_ineffability_implies_indescribability}, $S$ is $\Pi^1_{n+2}$-indescribable. Thus $S\in I^+$. Conversely, suppose $S\in I^+$. For the sake of contradiction, suppose $S\in\I(\Pi^1_n(\kappa))$. Then there is an $S$-list $\vec{S}=\<S_\alpha\st\alpha\in S\>$ such that every homogeneous set for $\vec{S}$ is in the ideal $\Pi^1_n(\kappa)$. This is expressible by a $\Pi^1_{n+2}$-sentence $\varphi$ over $(V_\kappa,\in,\vec{S})$. Thus, it follows that the set
\begin{align*}
C&=\{\alpha<\kappa\st(V_\alpha,\in,\vec{S}\cap V_\alpha)\models\varphi\}\\
	&=\{\alpha<\kappa\st \text{every hom. set for $\vec{S}\restrict\alpha$ is in $\Pi^1_n(\alpha)$}\}
\end{align*} 
is in the filter $\Pi^1_{n+2}(\kappa)^*$. Since $S\in I^+$, it follows that $S$ is not equal to the union of a non-subtle set and a non--$\Pi^1_{n+2}$-indescribable set. Since $S=(S\cap C)\cup(S\setminus C)$ and $S\setminus C\in\Pi^1_{n+2}(\kappa)$, it follows that $S\cap C$ must be subtle. Thus, by Lemma \ref{corollar_indescribable_hom_sets_from_subtlety}, there is some $\alpha\in S\cap C$ for which there is an $H\subseteq S\cap C\cap\alpha$ which is $\Pi^1_n$-indescribable in $\alpha$ and homogeneous for $\vec{S}$. This contradicts $\alpha\in C$.

For the remaining statement, let us show that if $\kappa\notin\I(\Pi^1_n(\kappa))$, then there are many cardinals below $\kappa$ which are both subtle and $\Pi^1_{n+2}$-indescribable. Suppose $\kappa\notin\I(\Pi^1_n(\kappa))$. Notice that the fact that $\kappa$ is subtle can be expressed by a $\Pi^1_1$-sentence $\varphi$ over $V_\kappa$ and thus the set
\[C=\{\alpha<\kappa\st (V_\alpha,\in)\models\varphi\}=\{\alpha<\kappa\st\text{$\alpha$ is subtle}\}\]
is in the filter $\Pi^1_1(\kappa)^*\subseteq\Pi^1_{n+2}(\kappa)^*$. Furthermore, by Lemma \ref{lemma_subtlety_is_stronger_than_indescribability}, the set
\[H=\{\alpha<\kappa\st\text{$\alpha$ is $\Pi^1_{n+2}$-indescribable}\}\]
is in the filter $\NSub_\kappa^*$. Since $\I(\Pi^1_n(\kappa))\supseteq\NSub_\kappa\cup\Pi^1_{n+2}(\kappa)$, it follows that $C\cap H$ is in the filter $\I(\Pi^1_n(\kappa))^*$.
\end{proof}

\begin{remark}\label{remark_iterating_ineffability}
Let us note that the preceding result of Baumgartner can be generalized to the ideals of the form $\I(\Pi^1_\xi(\kappa))$ where $\xi\geq\omega$ and $\Pi^1_\xi(\kappa)$ denotes the $\Pi^1_\xi$-indescribability ideal defined in Remark \ref{remark_transfinite_indescribability}. For example, if $\kappa\notin\I(\Pi^1_\omega(\kappa))$ then one has
\[\I(\Pi^1_{\omega}(\kappa))=\overline{\NSub_\kappa\cup\Pi^1_{\omega+2}(\kappa)}.\] Furthermore, by iterating the ineffability operator one obtains ideals of the form $\I^\alpha(\Pi^1_\xi(\kappa))$, and the previous Theorem of Baumgartner can also be generalized to these ideals. Considering the large cardinal notions associated to these ideals provides a strict refinement of Baumgartner's original ineffability hierarchy \cite{MR0540770}. This refinement of the ineffability hierarchy is analogous to the refinement of the Ramsey hierarchy studied in \cite{CodyRefinement}.
\end{remark}

Using the previous theorem of Baumgartner, we obtain another generic embedding characterization of almost ineffability (taking $n=-1$) and ineffability (taking $n=0$). Note that in Proposition \ref{proposition_ineffable_embedding_stationary}, the generic embedding was obtained by forcing with the nonstationary ideal below a particular stationary anticipation set, whereas in the following proposition, the generic embedding is obtained by forcing with the ineffable ideal below an ineffable set.

\begin{proposition}\label{proposition_ineffable_embedding_baumgartner}
For all $n\in\{-1\}\cup\omega$, for all cardinals $\kappa$ and all $S\subseteq\kappa$ the following are equivalent.
\begin{enumerate}
\item $S\in\I(\Pi^1_n(\kappa))^+$.
\item There is a generic elementary embedding $j:V\to M\subseteq V[G]$ with critical point $\kappa$ and $\kappa\in j(S)$ such that the following properties hold.
\begin{enumerate}
\item For all $A\in V_{\kappa+1}^V$ and all $\Pi^1_{n+2}$-sentences $\varphi$ we have
\[((V_\kappa,\in,A)\models\varphi)^V\implies((V_\kappa,\in,A)\models\varphi)^M.\]
\item For every $S$-list $\vec{S}=\<S_\alpha\st\alpha\in S\>$ in $V$, it follows that $M\models$ ``there is a set $H\subseteq S$  which is $\Pi^1_n$-indescribable in $\kappa$ and homogeneous for $\vec{S}$''.
\end{enumerate}
\end{enumerate}
\end{proposition}

\begin{proof}
For (1) implies (2), suppose $S\in\I(\Pi^1_n(\kappa))^+$. Let $G\subseteq P(\kappa)/(\I(\Pi^1_n(\kappa))\restrict S)$ be generic and let $j:V\to M=V^\kappa/U_G\subseteq V[G]$ be the corresponding generic ultrapower. Clearly, $\crit(j)=\kappa$ and $\kappa\in j(S)$. By Theorem \ref{theorem_baumgartner_ineffable_ideal}, we have $\Pi^1_{n+2}(\kappa)^*\subseteq\I(\Pi^1_n(\kappa))^*\subseteq U_G$, and thus (2)(a) follows by an argument similar to that of Proposition \ref{proposition_indescribability_embedding}. For (2)(b), fix an $S$-list $\vec{S}=\<S_\alpha\st\alpha\in S\>$. Since $S\in\I(\Pi^1_n(\kappa))^+$, there is a set $H\subseteq S$ which is $\Pi^1_n$-indescribable in $\kappa$ and homogeneous for $\vec{S}$. Clearly, $H=j(H)\cap\kappa\in M$, and furthermore, the fact that $H$ is $\Pi^1_n$-indescribable is expressible by a $\Pi^1_{n+1}$-sentence over $(V_\kappa,\in,H)$. Thus, by (2)(a), it follows that $M\models$ ``$H$ is $\Pi^1_n$-indescribable in $\kappa$ and homogeneous for $\vec{S}$''.

Suppose (2) holds. Let $\vec{S}=\<S_\alpha\st\alpha\in S\>$ be an $S$-list in $V$. We must show that $\vec{S}$ has a $\Pi^1_n$-indescribable homogeneous set. Suppose not. Recall that ``$X\in\Pi^1_n(\kappa)$'' is expressible by a $\Sigma^1_{n+1}$-sentence over $(V_\kappa,\in,X)$. Thus, there is a $\Pi^1_{n+2}$-sentence $\varphi$ over $(V_\kappa,\in,\vec{S})$ asserting that every homogeneous set for $\vec{S}$ is not $\Pi^1_n$-indescribable. Now let $j:V\to M\subseteq V[G]$ be a generic embedding as in (2). By (2)(a), it follows that $((V_\kappa,\in,\vec{S})\models\varphi)^M$ and hence $M\models$ ``every homogeneous set for $\vec{S}$ is not $\Pi^1_n$-indescribable''. This contradicts (2)(b).
\end{proof}

\subsection{The Ramsey ideal}\label{section_ramsey} Recall that $\kappa>\omega$ is a \emph{Ramsey cardinal} if for every function $f:[\kappa]^{<\omega}\to 2$ there is a set $H\subseteq\kappa$ of size $\kappa$ which is \emph{homogeneous} for $f$, meaning that $f\restrict[H]^n$ is constant for all $n<\omega$. Furthermore, for $S\subseteq\kappa$ where $\kappa$ is a cardinal, a function $f:[S]^{<\omega}\to\kappa$ is \emph{regressive} if $f(a)<\min a$ for all $a\in[S]^{<\omega}$. Given an ideal $I\supseteq[\kappa]^{<\kappa}$ on $\kappa$ we define another ideal $\R(I)$ on $\kappa$ by letting $S\notin\R(I)$ if and only if for every regressive function $f:[S]^{<\omega}\to \kappa$ there is a set $H\subseteq S$ in $I^+$ which is homogeneous for $f$. Feng proved \cite[Theorem 2.1]{MR1077260}, that $\R(I)$ is always a normal ideal. A set $S\subseteq\kappa$ is \emph{Ramsey} if $S\in\R([\kappa]^{<\kappa})^+$, that is every regressive function $f:[S]^{<\omega}\to \kappa$ has a homogeneous set $H\subseteq S$ of size $\kappa$. Baumgartner showed that when $\kappa$ is Ramsey, the collection $\R([\kappa]^{<\kappa})$ is a nontrivial normal ideal called the \emph{Ramsey ideal} on $\kappa$. In this section we will study the ideals $\R(\Pi^1_n(\kappa))$ for $n\in\{-1\}\cup\omega$.

In order to give a characterization of sets in $\R(\Pi^1_n(\kappa))^+$ for $n<\omega$ which is analogous to Proposition \ref{proposition_ineffability_anticipation_embedding}, let us present an alternative characterization of Ramseyness due to Feng. Indeed, Feng \cite[Theorem 2.3]{MR1077260} gave a characterization of Ramseyness which resembles the definition of ineffability. 

\begin{definition}[Feng \cite{MR1077260}] Suppose $S\subseteq\kappa$. For each $n<\omega$ and for all increasing sequences $\alpha_1<\cdots<\alpha_n$ taken from $S$ suppose that $S_{\alpha_1\ldots\alpha_n}\subseteq\alpha_1$. Then we say that
\[\vec{S}=\<S_{\alpha_1\ldots\alpha_n}\st n<\omega\land (\alpha_1,\ldots,\alpha_n)\in[S]^n\>\]
is an \emph{$(\omega,S)$-list}. A set $H\subseteq S$ is said to be \emph{homogeneous} for an $(\omega,S)$-lists $\vec{S}$ if for all $0<n<\omega$ and for all increasing sequences $\alpha_1<\cdots<\alpha_n$ and $\beta_1<\cdots<\beta_n$ taken from $S$ with $\alpha_1\leq\beta_1$ we have $S_{\alpha_1\cdots\alpha_n}=S_{\beta_1\cdots\beta_n}\cap\alpha_1$.
\end{definition}


\begin{theorem}[Feng \cite{MR1077260}]\label{theorem_omega_S}
Let $\kappa$ be a regular cardinal and suppose $I$ is an ideal on $\kappa$ such that $I\supseteq \NS_\kappa$. For $S\subseteq \kappa$ the following are equivalent.
\begin{enumerate}
\item $S\in\R(I)^+$, that is, every function $f:[S]^{<\omega}\to 2$ has a homogeneous set $H\in P(S)\cap I^+$.
\item For all $\gamma<\kappa$, every function $f:[S]^{<\omega}\to\gamma$ has a homogeneous set $H\in P(S)\cap I^+$.
\item Every structure $\mathcal{A}$ in a language of size less than $\kappa$ with $\kappa\subseteq\mathcal{A}$ has a set of indiscernibles $H\in P(S)\cap I^+$.
\item $S\in\R(I)^+$, that is, for every regressive function $f:[S]^{<\omega}\to\kappa$ there is a set $H\in P(S)\cap I^+$ which is homogeneous for $f$.
\item For every club $C\subseteq\kappa$, every regressive function $f:[S]^{<\omega}\to\kappa$ has a homogeneous set $H\in P(S\cap C)\cap I^+$.
\item For all $(\omega,S)$-sequences $\vec{S}$ there is a set $H\in P(S)\cap I^+$ which is homogeneous for $\vec{S}$.
\item For all $(\omega,S)$-lists $\vec{S}$ there is a set $H\subseteq S$ with $H\in I^+$ and a sequence $\<D_n\st n<\omega\>$ such that for each $n<\omega$ and for all $\alpha_1<\cdots<\alpha_n$ from $H$ we have $S_{\alpha_1\cdots\alpha_n}=D_n\cap\alpha_1$.
\end{enumerate}
\end{theorem}

In the previous theorem, if one weakens the assumption $I\supseteq\NS_\kappa$ to $I\supseteq [\kappa]^{<\kappa}$ one can still prove \emph{some} of the equivalences. For more on this issue see \cite{CodyRefinement}.

\begin{proposition}
Suppose $I\supseteq[\kappa]^{<\kappa}$ is an ideal on a regular cardinal $\kappa$. Then clauses (4), (5), (6) and (7) of Theorem \ref{theorem_omega_S} are equivalent.
\end{proposition}

\begin{remark}
Although the ideals $\I(\Pi^1_n(\kappa))$ for $n\in\{-1\}\cup\kappa$ as well as the ideals $\R(\Pi^1_{-1}(\kappa))=\R([\kappa]^{<\kappa})$ and $\R(\Pi^1_0(\kappa))=\R(\NS_\kappa)$, have been well-studied (see \cite{MR0384553}, \cite{MR0540770} and \cite{MR1077260}), ideals of the form $\R(\Pi^1_n(\kappa))$ for $n>1$ and the corresponding large cardinal properties have not been studied until recently \cite{CodyRefinement}.
\end{remark}

\begin{remark}\label{remark_trivial}
Let $\kappa$ be a cardinal and suppose $I\supseteq\NS_\kappa$ is a normal ideal on $\kappa$. Suppose $S\in\R(I)^+$ and let $\vec{S}=\<S_{\alpha_1\ldots\alpha_n}\st n<\omega\land (\alpha_1,\ldots,\alpha_n)\in[S]^n\>$ be an $(\omega,S)$-list. Let $H\subseteq S$ and $\<D_n\st n<\omega\>$ be as in Theorem \ref{theorem_omega_S}(7). Then for each $n<\omega$ we have
\[H\subseteq\{\alpha\in S\st (\forall\alpha_2\ldots\alpha_n\in H)(\alpha<\alpha_2<\cdots<\alpha_n\implies S_{\alpha\alpha_2\cdots\alpha_n}=D_n\cap\alpha)\}.\]
\end{remark}

The set which contains $H$ in the statement of Remark \ref{remark_trivial} can be thought of as the set $X$ of ordinals at which the sequence $\<D_n\st n<\omega\>$ is anticipated by $\vec{S}$. We obtain the following generic embedding characterization of $\R(\Pi^1_n(\kappa))$ by forcing with $P(\kappa)/\Pi^1_n(\kappa)$ below the `anticipation set' $X$.

\begin{theorem}\label{theorem_ramsey_anticipation_embedding}
For all $n<\omega$, all cardinals $\kappa$ and all $S\subseteq\kappa$, the following are equivalent.
\begin{enumerate}
\item $S\in\R(\Pi^1_n(\kappa))^+$
\item For all $(\omega,S)$-lists $\vec{S}$ there is a set $H\subseteq\kappa$ and there is a generic elementary embedding $j:V\to M\subseteq V[G]$ with critical point $\kappa$ and $\kappa\in j(H)$ such that the following properties hold.\footnote{Note that since $\kappa\in j(H)$ it follows that $H$ is stationary in $\kappa$, and more.}
\begin{enumerate}
\item For any $\<a_m\st m<\omega\>\in\left(\prod_{1<m<\omega}[j(H)]^m\right)^{V[G]}$ such that $a_m(0)=\kappa$ for all $m<\omega$, we have 
\[\<j(\vec{S})(a_m)\st m<\omega\>\in V.\]
\item For all $A\in V_{\kappa+1}^V$ and all $\Pi^1_n$-sentences $\varphi$ we have
\[((V_\kappa,\in,A)\models\varphi)^V\implies((V_\kappa,\in,A)\models\varphi)^M.\]
\end{enumerate}
\end{enumerate}
\end{theorem}

\begin{proof}
Suppose $S\in\R(\Pi^1_n(\kappa))^+$ and let $\vec{S}$ be an $(\omega,S)$-list. Let $H\subseteq S$ and $\<D_m\st m<\omega\>$ be the sets obtained from Theorem \ref{theorem_omega_S}(7). Since $H\in \Pi^1_n(\kappa)^+$, the ideal $\Pi^1_n(\kappa)\restrict H$ is normal. Let $G\subseteq P(\kappa)/(\Pi^1_n(\kappa)\restrict H)$ be generic and let $j:V\to M=V^\kappa/U_G\subseteq V[G]$ be the corresponding generic ultrapower embedding. Then $\crit(j)=\kappa$ and $\kappa\in j(H)\subseteq j(S)$. To prove that (2)(a) holds, fix any $\<a_m\st m<\omega\>\in(\prod_{1<m<\omega}[j(H)]^m)^{V[G]}$ such that $a_m(0)=\kappa$. By Remark \ref{remark_trivial}, the fact that $\kappa\in j(H)$ implies that $D_m=j(\vec{S})(a_m)\cap\kappa$ for all $m<\omega$, and thus it follows that $\<j(\vec{S})(a_m)\st m<\omega\>\in V$. To see that (2)(b) holds, fix any $A\in V_{\kappa+1}^V$ and any $\Pi^1_n$-sentence $\varphi$ such that $((V_\kappa,\in,A)\models\varphi)^V$. Then $C=\{\alpha<\kappa\st(V_\alpha,\in,A\cap V_\alpha)\models\varphi\}\in\Pi^1_n(\kappa)^*\subseteq U_G$ and hence $\kappa\in j(C)$, which implies $((V_\kappa,\in,A)\models\varphi)^M$.

Conversely, suppose (2) holds. To see that $S\in\R(\Pi^1_n(\kappa))^+$, fix an $(\omega,S)$-list $\vec{S}\in V$. Let $H\in P(S)\cap V$ and $j:V\to M\subseteq V[G]$ be as given in (2). Select any $\<a_m\st m<\omega\>\in (\prod_{1<m<\omega}[j(H)]^m)^{V[G]}$ with $a_m(0)=\kappa$ for all $m<\omega$, and define $D_m=j(\vec{S})(a_m)$. By (2)(a) we have $\<D_m\st m<\omega\>\in V$. We define
\begin{align*}
H'=\{\alpha\in H\st &(\forall m\in\omega\setminus 2)(\forall\alpha_2\cdots\alpha_m\in H) \\
	&(\alpha<\alpha_2<\cdots<\alpha_m\implies S_{\alpha\alpha_2\cdots\alpha_m}=D_m\cap\alpha\}.
\end{align*}
Let us check that $H'$ is $\Pi^1_n$-indescribable in $\kappa$ and homogeneous for $\vec{S}$. Suppose $R\in V_{\kappa+1}^V$ and $\varphi$ is a $\Pi^1_n$-sentence such that $((V_\kappa,\in,R)\models\varphi)^V$. By (2)(b) we have $((V_\kappa,\in,R)\models\varphi)^M$. Furthermore, from (2)(a) and the definition of $D_m$ we see that \[j(\vec{S})_{\kappa\kappa_2\ldots\kappa_m}=j(D_m)\cap\kappa\]
holds for $m\in\omega\setminus 2$ and all $\kappa_2<\ldots<\kappa_m$ from $j(H)\setminus(\kappa+1)$. Thus $\kappa\in j(H')$ and  therefore, by elementarity there is some $\alpha\in H'$ such that $(V_\alpha,\in,R\cap V_\alpha)\models\varphi$. To check that $H'$ is homogeneous for $\vec{S}$, fix $m<\omega$ and let $\beta_1<\cdots<\beta_m$ and $\gamma_1<\cdots<\gamma_m$ be elements of $H'$ with $\beta_1\leq\gamma_1$. Then since $\beta_1\in H'$ and $\beta_2,\ldots,\beta_m\in H$ we have $S_{\beta_1\cdots\beta_m}=D_m\cap\beta_1$. Similarly, $S_{\gamma_1\cdots\gamma_m}=D_m\cap\gamma_1$. Since $\beta_1\leq\gamma_1$ we have $S_{\beta_1\cdots\beta_m}=S_{\gamma_1\cdots\gamma_m}\cap\beta_1$.
\end{proof}

In order to give a second generic embedding characterization of sets which are positive for the ideal $\R(\Pi^1_n(\kappa))$, where $n\in\{-1\}\cup\omega$, we will prove a generalization of results due to Baumgartner \cite[Theorem 4.4]{MR0540770} and Feng \cite[Theorem 4.8]{MR1077260}. In the study of the Ramsey ideal $\R([\kappa]^{<\kappa})$, the notion of pre-Ramseyness plays a role which is analogous to that of subtlety in the study of ineffability. A set $S\subseteq\kappa$ is \emph{pre-Ramsey} if and only if for every regressive function $f:[S]^{<\omega}\to \kappa$ and every club $C\subseteq\kappa$ there is some $\alpha\in S\cap C$ such that there is a set $H\subseteq S\cap C\cap \alpha$ which has size $\alpha$ and is homogeneous for $f$. Baumgartner proved \cite{MR0540770} that if $\kappa$ is pre-Ramsey then the collection 
\[\text{NPreRam}_\kappa=\{X\subseteq\kappa\st\text{$X$ is not pre-Ramsey}\}\]
is a nontrivial normal ideal on $\kappa$ called the \emph{pre-Ramsey ideal on $\kappa$}. In order to generalize the results of Baumgartner and Feng, we will need a generalization of pre-Ramseyness. 

Let us define an operator $\R_0$ which maps ideals on a cardinal $\kappa$ to ideals on $\kappa$. Suppose that $\vec{I}=\<I_\alpha\st\text{$\alpha\leq\kappa$ is a cardinal}\>$ is a sequence such that each $I_\alpha\supseteq[\alpha]^{<\alpha}$ is an ideal on $\alpha$. We define an ideal $\R_0(\vec{I})$ on $\kappa$ by letting $S\in\R_0(\vec{I})^+$ if and only if for every regressive function $f:[S]^{<\omega}\to\kappa$ and every club $C\subseteq\kappa$ there is an $\alpha\in S\cap C$ for which there is a set $H\subseteq S\cap C\cap\alpha$ in $I_\alpha^+$ which is homogeneous for $f$. It should be understood that many of the ideals $I_\alpha$ will be trivial, and when no confusion will arise, as in the case where the ideals $I_\alpha$ have a uniform definition, we write $\R_0(I_\kappa)$ instead of $\R_0(\vec{I})$. Thus a set $S\subseteq\kappa$ is \emph{pre-Ramsey} if and only if $S\in\R_0([\kappa]^{<\kappa})^+$.

We are now ready to present the generalization of the results of Feng and Baumgartner. The case $n=-1$ (i.e. $\Pi^1_{-1}(\kappa)=[\kappa]^{<\kappa}$) is due to Baumgartner, the case $n=0$ (i.e. $\Pi^1_0(\kappa)=\NS_\kappa$) is due to Feng and the remaining cases are due to the author (this theorem is a special case of \cite[Theorem 6.1]{CodyRefinement}).

\begin{theorem}\label{theorem_ramsey_ideals}
For all $n\in\{-1\}\cup\omega$ and all cardinals $\kappa$, we have $\kappa\notin\R(\Pi^1_n(\kappa))$ if and only if the following properties hold.
\begin{enumerate}
\item $\kappa\in\R_0(\Pi^1_n(\kappa))^+$ and $\kappa$ is $\Pi^1_{n+2}$-indescribable.
\item The ideal generated by $\R_0(\Pi^1_n(\kappa))$ and the $\Pi^1_{n+2}$-indescribable ideal is a nontrivial normal ideal; in this case, \[\R(\Pi^1_n(\kappa))=\overline{\R_0(\Pi^1_n(\kappa))\cup\Pi^1_{n+2}(\kappa)}.\]
\end{enumerate}
Furthermore, (2) cannot be removed from this characterization because the least cardinal $\kappa$ such that $\kappa\in\R_0(\Pi^1_n(\kappa))^+$ and $\kappa$ is $\Pi^1_{n+2}$-indescribable is in $\R(\Pi^1_n(\kappa))$.
\end{theorem}

\begin{proof}
Let $I=\overline{\R_0(\Pi^1_n(\kappa))\cup\Pi^1_{n+2}(\kappa)}$. We will show that $S\in\R(\Pi^1_n(\kappa))^+$ if and only if $S\in I^+$.

Suppose $S\in \R(\Pi^1_n(\kappa))^+$. To show that $S\in I^+$, it suffices to show that $S\in \Pi^1_{n+2}(\kappa)^+$ and $S\in\R_0(\Pi^1_n(\kappa))^+$ (see Remark \ref{remark_ideal_containment}). Let $\vec{S}=\<S_\alpha\st\alpha\in S\>$ be an $S$-list. By arbitrarily extending $\vec{S}$ to an $(\omega,S)$-list, we see that the fact that $S\in\R(\Pi^1_n(\kappa))^+$ implies that the $S$-list $\vec{S}$ has a homogeneous $H\subseteq S$ which is $\Pi^1_n$-indescribable in $\kappa$. Thus, by Lemma \ref{lemma_ineffability_implies_indescribability}, it follows that $S\in\Pi^1_{n+2}(\kappa)^+$. To see that $S\in\R_0(\Pi^1_n(\kappa))^+$, fix a regressive function $f:[S]^{<\omega}\to\kappa$ and a club $C\subseteq\kappa$. Since $S\cap C\in\R(\Pi^1_n(\kappa))^+$, there is set $H\subseteq S\cap C$ in $\Pi^1_n(\kappa)^+$ which is homogeneous for $f$. The fact that $H\in\Pi^1_n(\kappa)^+$ can be expressed by a $\Pi^1_{n+1}$ sentence $\varphi$ over $(V_\kappa,\in,H)$. Since $S\cap C$ is $\Pi^1_{n+2}$-indescribable, it follows that there is an $\alpha\in S\cap C$ such that $(V_\alpha,\in,H\cap V_\alpha)\models\varphi$, which implies that $H\cap\alpha$ is $\Pi^1_n$-indescribable in $\alpha$. Thus $S\in \R_0(\Pi^1_n(\kappa))^+$.

Suppose $S\in I^+$. To show that $S\in\R(\Pi^1_n(\kappa))$ fix a regressive function $f:[S]^{<\omega}\to \kappa$ and suppose, for the sake of contradiction, that every homogeneous set for $f$ is not $\Pi^1_n$-indescribable in $\kappa$. This can be expressed by a
$\Pi^1_{n+2}$ sentence $\varphi$ over $(V_\kappa,\in,S,f)$. Thus, the set
\begin{align*}
C&=\{\alpha<\kappa\st (V_\alpha,\in,S\cap\alpha,f\cap V_\alpha)\models\varphi\}\\
	&=\{\alpha<\kappa\st \text{every hom. set $H\subseteq S\cap\alpha$ for $f$ is in $\Pi^1_n(\alpha)$}\}
\end{align*}
is in the filter $\Pi^1_{n+2}(\kappa)^*$. Since $S\in I^+$ it follows that $S$ is not the union of a set in $\R_0(\Pi^1_n(\kappa))$ and a set in $\Pi^1_{n+2}(\kappa)$. Since $S=(S\cap C)\cup (S\setminus C)$ and $S\setminus C\in\Pi^1_{n+2}$, we see that $S\cap C\in\R_0(\Pi^1_n(\kappa))^+$. Hence, there is an $\alpha\in S\cap C$ such that there is an $H\subseteq S\cap C\cap\alpha$ which is $\Pi^1_n$-indescribable in $\alpha$ and homogeneous for $f$. This contradicts $\alpha\in C$.

For the remaining statement, let us show that if $\kappa\notin\R(\Pi^1_n(\kappa))$, then there are many cardinals $\alpha<\kappa$ such that $\alpha\in\R_0(\Pi^1_n(\alpha))^+$ and $\alpha\in\Pi^1_{n+2}(\alpha)^+$. Suppose $\kappa\notin\R(\Pi^1_n(\kappa))$. Notice that the fact that $\kappa\in\R_0(\Pi^1_n(\kappa))^+$ can be expressed by a $\Pi^1_1$-sentence $\varphi$ over $V_\kappa$ and thus the set
\[C=\{\alpha<\kappa\st (V_\alpha,\in)\models\varphi\}=\{\alpha<\kappa\st\text{$\alpha\in\R_0(\Pi^1_n(\kappa))^+$}\}\]
is in the filter $\Pi^1_1(\kappa)^*\subseteq\Pi^1_{n+2}(\kappa)^*$. Furthermore, by Lemma \ref{lemma_subtlety_is_stronger_than_indescribability}, the set
\[H=\{\alpha<\kappa\st\text{$\alpha$ is $\Pi^1_{n+2}$-indescribable}\}\]
is in the filter $\NSub_\kappa^*\subseteq\R_0(\Pi^1_n(\kappa))^*$. Since $\R(\Pi^1_n(\kappa))^*\supseteq \R_0(\Pi^1_n(\kappa))^*\cup\Pi^1_{n+2}(\kappa)^*$, it follows that $C\cap H$ is in the filter $\R(\Pi^1_n(\kappa))^*$.
\end{proof}

\begin{remark}\label{remark_preramsey_question}
Comparing the statements of Theorem \ref{theorem_baumgartner_ineffable_ideal} and \ref{theorem_ramsey_ideals}, one would like to strengthen Theorem \ref{theorem_ramsey_ideals} by replacing $\R_0(\Pi^1_n(\kappa))$ with the pre-Ramsey ideal $\NPreRam_\kappa$. However, it seems to be unknown whether or not this is possible. The problem is that it is not known whether Lemma \ref{lemma_subtlety_is_stronger_than_indescribability} can be generalized to the pre-Ramsey ideal. See Question \ref{question_pi1n_ramsey_ideal} and Question \ref{question_generalize_subtle_result} in Section \ref{section_questions} below.
\end{remark}

\begin{remark}\label{remark_iterating_ineffability}
Let us note that the preceding result of Baumgartner can be generalized to the ideals of the form $\R(\Pi^1_\xi(\kappa))$ where $\xi\geq\omega$ and $\Pi^1_\xi(\kappa)$ denotes the $\Pi^1_\xi$-indescribability ideal defined in Remark \ref{remark_transfinite_indescribability}. Furthermore, by iterating the Ramsey operator one obtains ideals of the form $\R^\alpha(\Pi^1_\beta(\kappa))$, and the previous Theorem of Baumgartner can also be generalized to these ideals. For example, if $\kappa\in\R^m(\Pi^1_\beta(\kappa))^+$, one has
\[\R^m(\Pi^1_\beta(\kappa))=\overline{\R_0(\R^{m-1}(\Pi^1_\beta(\kappa)))\cup\Pi^1_{\beta+2m}(\kappa)}.\] See \cite{CodyRefinement} for more details.
\end{remark}

Next we give a second generic embedding characterization of sets which are positive for the ideal $\R(\Pi^1_n(\kappa))$. Taking $n=-1$ yields a generic embedding characterization of Ramsey cardinals (and sets). The following is a special case of \cite[Theorem 8.4]{CodyRefinement}.

\begin{theorem}\label{theorem_ramsey_embedding_baumgartner}
For all $n\in\{-1\}\cup\omega$, for all cardinals $\kappa$ and all $S\subseteq\kappa$ the following are equivalent.
\begin{enumerate}
\item $S\in\R(\Pi^1_n(\kappa))^+$.
\item There is a generic elementary embedding $j:V\to M\subseteq V[G]$ with critical point $\kappa$ and $\kappa\in j(S)$ such that the following properties hold.
\begin{enumerate}
\item For all $A\in V_{\kappa+1}^V$ and all $\Pi^1_{n+2}$-sentences $\varphi$ we have
\[((V_\kappa,\in,A)\models\varphi)^V\implies((V_\kappa,\in,A)\models\varphi)^M.\]
\item For every regressive function $f:[S]^{<\omega}\to \kappa$ in $V$, it follows that $M\models$ ``there is a set $H\subseteq S$  which is $\Pi^1_n$-indescribable in $\kappa$ and homogeneous for $f$''.
\end{enumerate}
\end{enumerate}
\end{theorem}

\begin{proof}
The proof is very similar to that of Proposition \ref{proposition_ineffable_embedding_baumgartner}. For (1) implies (2), suppose $S\in\R(\Pi^1_n(\kappa)^+$. Let $G\subseteq P(\kappa)/(\R(\Pi^1_n(\kappa)))\restrict S)$ be generic and let $j:V\to M=V^\kappa/U_G\subseteq V[G]$ be the corresponding generic ultrapower. Since $\Pi^1_{n+2}(\kappa)^*\subseteq\R(\Pi^1_n(\kappa))^*\subseteq U_G$, (2)(a) holds (see the proof of Proposition \ref{proposition_indescribability_embedding} or Proposition \ref{proposition_ineffable_embedding_baumgartner}). Fix a regressive function $f:[S]^{<\omega}\to \kappa$ in $V$. Since $S\in\R(\Pi^1_n(\kappa))^+$ there is a set $H\subseteq S$ in $V$ which is $\Pi^1_n$-indescribable in $\kappa$ and homogeneous for $f$. Then $H=j(H)\cap\kappa\in M$ and by (2)(a) together with the fact that the $\Pi^1_n$-indescribability of $H$ is expressible by a $\Pi^1_{n+1}$ sentence over $(V_\kappa,\in,H)$, we conclude that $M\models$ ``$H\subseteq S$ is $\Pi^1_n$-indescribable and homogeneous for $f$''.

For (2) implies (1), suppose there is a generic embedding $j:V\to M\subseteq V[G]$ as in (2). To show that $S\in\R(\Pi^1_n(\kappa))^+$, fix a regressive function $f:[S]^{<\omega}\to \kappa$. For the sake of contradiction assume every homogeneous set for $f$ is in $\Pi^1_n(\kappa)$. Recall that ``$X\in\Pi^1_n(\kappa)$'' is expressible by a $\Sigma^1_{n+1}$ sentence over $(V_\kappa,\in,X)$. Thus there is a $\Pi^1_{n+2}$ sentence $\varphi$ over $(V_\kappa,\in,f)$ asserting that every homogeneous set for $f$ is not $\Pi^1_n$-indescribable. By (2)(a), it follows that $((V_\kappa,\in,f)\models\varphi)^M$ and hence $M\models$ ``every homogeneous set for $f$ is not $\Pi^1_n$-indescribable''. This contradicts (2)(b).
\end{proof}



\section{Splitting positive sets assuming weak compactness}\label{section_splitting}

In this section we present a folklore result which states that assuming $\kappa$ is weakly compact, many large cardinal ideals on $\kappa$ are nowhere $\kappa$-saturated. We put together techniques used by Hellsten \cite[Theorem 2]{MR2653962} in the context of $\Pi^1_n$-indescrib\-ability and Brickhill-Welch \cite{BrickhillWelch} in the context of $\Pi^1_\gamma$-indescribability, and note that previously known methods allow for more general conclusions than what may have been known. Hellsten \cite{MR2653962} attributes the following result to Tarski \cite{MR17737}.

\begin{theorem}[Tarski]\label{theorem_tarski}
If $\kappa$ is a weakly compact cardinal and $I$ is a $\kappa$-complete ideal on $\kappa$ such that for every $S\in I^+$ there are $S_0,S_1\in I^+$ such that $S=S_0\sqcup S_1$, then $I$ is nowhere $\kappa$-saturated.
\end{theorem}

\begin{proof} Suppose $S\in I^+$. We will show that $I\restrict S$ is not $\kappa$-saturated. We define a tree $(T,\supseteq)$ where $T\subseteq I^+\cap P(S)$ as follows. We will inductively define $\Sigma_\alpha\subseteq\! ^\alpha 2$ and $\Lev_\alpha(T)\subseteq I^+$ such that for each $ j\in\Sigma_\alpha$ there is some $S_ j\in \Lev_\alpha(T)$ and $\text{Lev}_\alpha(T)=\{s_j\st j\in \Sigma_\alpha\}$. Let $\Sigma_0=\{\emptyset\}$ and $S_\emptyset=S$. Suppose $\Sigma_\alpha$ and $\Lev_\alpha(T)$ have been defined. For each $ j\in\Sigma_\alpha$ we have $S_ j\in\Lev_\alpha(T)\subseteq I^+$. Let $S_{ j\concat 0},S_{ j\concat 1}\in I^+$ be such that $S_ j=S_{ j\concat 0}\sqcup S_{ j\concat 1}$ and add $ j\concat 0, j\concat 1$ to $\Sigma_{\alpha+1}$. In other words, $\Sigma_{\alpha+1}=\{ j\concat i\st  j\in\Sigma_\alpha\land i=0,1\}$. Now, suppose $\Sigma_\alpha$ and $\Lev_\alpha(T)$ have been defined for all $\alpha<\eta$ where $\eta$ is a limit ordinal. For each $ j\in\!^\eta 2$, if $ j\restrict\alpha\in\Sigma_\alpha$ for all $\alpha<\eta$ and $\bigcap_{\alpha<\eta}S_{ j\restrict\alpha}\in I^+$, then let $S_ j=\bigcap_{\alpha<\eta}S_{ j\restrict\alpha}$ and add $ j$ to $\Sigma_\eta$. This completes the definition of the tree $(T,\supseteq)$. It is relatively straight forward to show that $(S,\supseteq)$ is a $\kappa$-tree, and thus, applying the weak compactness of $\kappa$, must have a cofinal branch $b\subseteq T$, which provides a partition of $S$ into $\kappa$ disjoint $I^+$-sets.
\end{proof}

The following result is attributed to L\'evy and Silver in \cite{MR495118}.

\begin{corollary}[L\'evy-Silver]
If $\kappa$ is weakly compact and not measurable then every normal ideal on $\kappa$ is nowhere $\kappa$-saturated.
\end{corollary}

\begin{proof}
Suppose $\kappa$ is weakly compact and not measurable. Fix a normal ideal $I$ on $\kappa$. By Theorem \ref{theorem_tarski}, it suffices to show that for every $S\in I^+$ there are $S_0,S_1\in I^+$ such that $S=S_0\sqcup S_1$. Suppose $S\in I^+$ does not split. Then
\[U:=(I\restrict S)^*=\{X\subseteq\kappa\st(\kappa\setminus X)\cap S\in I\}\]
is a normal measure on $\kappa$, a contradiction.
\end{proof}

\begin{corollary}[Folklore]\label{corollary_folklore_splitting}
If $\kappa$ is weakly compact and $I$ is a normal ideal which is definable over $H_\kappa^+$, then $I$ is nowhere $\kappa$-saturated.\footnote{The author would like to thank Sean Cox for pointing out the statement and proof of this result.}
\end{corollary}

\begin{proof}
Suppose $\kappa$ is weakly compact and $I$ is a normal ideal on $\kappa$ which is definable over $H_{\kappa^+}$. By Theorem \ref{theorem_tarski}, it suffices to show that every $I$-positive set splits. Choose $S\in I^+$ and suppose $S$ does not split. Then $(I\restrict S)^*$ is a normal measure on $\kappa$. Let $j:V\to M$ be the ultrapower by $(I\restrict S)^*$. Then $H_{\kappa^+}=H_{\kappa^+}^M$ which implies $(I\restrict S)^*\in M$, which is a contradiction since a normal measure cannot be an element of its own ultrapower.
\end{proof}

Since the assumption that a cardinal $\kappa$ is $\Pi^1_n$-indescribable, almost ineffable or Ramsey implies that $\kappa$ is weakly compact, the next corollary follows directly from Corollary \ref{corollary_folklore_splitting}. Let us note that Hellsten attributes Corollary \ref{corollary_hellsten_splitting}(1) to Solovay and Tarski (see the end of Section 1 in \cite{MR2653962}).

\begin{corollary}\label{corollary_hellsten_splitting}
The following hold.
\begin{enumerate}
\item (Hellsten \cite{MR2653962}) For $n<\omega$, if $\kappa$ is $\Pi^1_n$-indescribable then the $\Pi^1_n$-indescribable ideal on $\kappa$ is nowhere $\kappa$-saturated.
\item If $\kappa$ is (almost) ineffable then the (almost) ineffable ideal on $\kappa$ is nowhere $\kappa$-saturated.
\item If $\kappa$ is Ramsey then the Ramsey ideal on $\kappa$ is nowhere $\kappa$-saturated. 
\item If $\kappa$ is weakly compact and subtle then the subtle ideal on $\kappa$ is nowhere $\kappa$-saturated.
\item For $n<\omega$, if $\kappa$ is $\Pi^1_n$-indescribable then the $\Pi^1_n$-indescribable ideal on $\kappa$ is nowhere $\kappa$-saturated.
\end{enumerate}
\end{corollary}

See Section \ref{section_questions} for several questions related to Corollary \ref{corollary_hellsten_splitting} which appear to remain open.

The next result can be used to show that if a cardinal $\kappa$ satisfies a large enough large cardinal property, then many ideals on $\kappa$ will not be $\kappa^+$-saturated. This result bears some similarity to work of Leary \cite{MR2634284} on ideal families and is implicit in the work of Hellsten \cite{MR2252250}.

\begin{proposition}[Hellsten \cite{MR2252250}]\label{proposition_general_splitting}
Suppose that for each $\alpha\leq\kappa$, $I_\alpha$ is a normal ideal on $\alpha$ and let $I=\<I_\alpha\st\alpha\leq\kappa\>$.\footnote{Here we allow for ideals to be trivial. For example, when $\alpha$ is singular we have $I_\alpha=P(\alpha)$.} We define the trace operation $\Tr_I:P(\kappa)\to P(\kappa)$ by $\Tr_I(X)=\{\alpha<\kappa\st X\cap\alpha\in I_\alpha^+\}$ and suppose that the following conditions hold.
\begin{enumerate}
\item For every $S\in I_\kappa^+$ we have $S\setminus \Tr_I(S)\in I_\kappa^+$.
\item Suppose there is a normal ideal $J\supseteq I_\kappa$ on $\kappa$ and there is an $A\in J^+$ such that the filter $(J\restrict A)^*$ is closed under $\Tr_I$.
\end{enumerate}
Then $I_\kappa\restrict A$ is not $\kappa^+$-saturated.
\end{proposition}

\begin{proof}
Suppose $J\supseteq I_\kappa$ is a normal ideal on $\kappa$ and $A\in J^+$ is such that $(J\restrict A)^*$ is closed under $\Tr_I$. Suppose $I_\kappa\restrict A$ is $\kappa^+$-saturated. Recall that by \cite[Theorem 3.1]{MR0505505}, an ideal $I$ on $\kappa$ is $\kappa^+$-saturated if and only if the ideals $I\restrict S$ for $S\in I^+$ are the only normal ideals on $\kappa$ which extend $I$. Therefore, since $I_\kappa\restrict A\subseteq J\restrict A$ and $J\restrict A$ is a normal ideal, it follows that $J\restrict A=(I_\kappa\restrict A)\restrict B$ for some $B\in (I_\kappa\restrict A)^+$. Thus $J\restrict A=I_\kappa\restrict (A\cap B)$. Now $A\cap B\in I_\kappa^+$ which implies $(A\cap B)\setminus \Tr_I(A\cap B)\in I_\kappa^+$ by assumption. Furthermore, since $A\cap B\in (J\restrict A)^*$, it follows that $\Tr_I(A\cap B)\in (J\restrict A)^*$, which implies $\kappa\setminus\Tr_I(A\cap B)\in J\restrict A=I_\kappa\restrict(A\cap B)$ and hence $(A\cap B)\setminus\Tr_I(A\cap B)\in I_\kappa$, a contradiction.
\end{proof}

From Proposition \ref{proposition_general_splitting} we can easily derive the following result of Hellsten \cite{MR2252250}.

\begin{corollary}[Hellsten \cite{MR2252250}]
If $\kappa$ is $\kappa^+$-$\Pi^1_n$-indescribable (i.e. $\Tr_{\Pi^1_n(\kappa)}^\alpha(\kappa)\in\Pi^1_n(\kappa)^+$ for all $\alpha<\kappa^+$) then the $\Pi^1_n$-indescribable ideal on $\kappa$ is not $\kappa^+$-saturated.
\end{corollary}

\section{Consistency results}\label{section_consistency_results}

Recently progress has been made in generalizing forcing constructions from the nonstationary ideal to the weakly compact ideal. Such generalizations often require a substantial amount of effort because there are some significant differences between stationarity and weak compactness: for example, one cannot force ground model nonstationary sets to become stationary in an extension, whereas, in some settings, ground model non--weakly compact sets can become weakly compact in after certain kinds of forcing \cite{MR495118}. The main reasons for the success in this direction are twofold. First, weak compactness has a simple characterization in terms of elementary embeddings which allows one to use standard techniques to argue that weak compactness is preserved after certain Easton-support forcing iterations (see \cite{MR2768691}). Second, weakly compact sets have a characterization in terms of $1$-clubs which resembles the definition of stationarity, so constructions which involve club shooting forcings can sometimes be generalized to forcings which shoot $1$-clubs. In the current section we briefly discuss these characterizations of weakly compact sets as well as some forcing constructions concerning the weakly compact ideal. We also discuss related characterizations of $\Pi^1_n$-indescribability and state a few open questions.

\subsection{$n$-clubs and indescribability embeddings}

A transitive set $M\models\ZFC^-$ of size $\kappa$ with $\kappa\in M$ and $M^{<\kappa}=\kappa$ is called a \emph{$\kappa$-model}. The following folklore result follows easily from \cite[Section 2]{MR0540770}.

\begin{lemma}
For all cardinals $\kappa$ and $S\subseteq\kappa$ the following are equivalent.
\begin{enumerate}
\item $S$ is $\Pi^1_1$-indescribable.
\item For all $A\subseteq\kappa$ there are $\kappa$-models $M$ and $N$ with $A,S\in M$ such that there is an elementary embedding $j:M\to N$ with critical point $\kappa$ and $\kappa\in j(S)$.
\end{enumerate}
\end{lemma}

\begin{definition}[Hauser \cite{MR1133077}]
Suppose $\kappa$ is inaccessible. For $n\geq 0$, a $\kappa$-model $N$ is \emph{$\Pi^1_n$-correct at $\kappa$} if and only if
\[V_\kappa\models\varphi \iff (V_\kappa\models\varphi)^N\]
for all $\Pi^1_n$-formulas $\varphi$ whose parameters are contained in $N\cap V_{\kappa+1}$.
\end{definition}

\begin{remark}
	Notice that every $\kappa$-model is $\Pi^1_0$-correct at $\kappa$.
\end{remark}

An easy adaptation of the arguments in \cite{MR1133077} establishes the following.

\begin{theorem}[Hauser \cite{MR1133077}]\label{theorem_hauser}
The following statements are equivalent for every inaccessible cardinal $\kappa$, every subset $S \subseteq \kappa$, and all $0<n<\omega$.
\end{theorem}

\begin{enumerate}
\item $S$ is $\Pi^1_n$-indescribable.
\item For every $A\subseteq\kappa$ there is a $\kappa$-model $M$ with $A,S\in M$ for which there is a $\Pi^1_{n-1}$-correct $\kappa$-model $N$ and an elementary embedding $j:M\to N$ with $\crit(j)=\kappa$ such that $\kappa\in j(S)$.
\item For every $A\subseteq\kappa$ there is a $\kappa$-model $M$ with $A,S\in M$ for which there is a $\Pi^1_{n-1}$-correct $\kappa$-model $N$ and an elementary embedding $j:M\to N$ with $\crit(j)=\kappa$ such that $\kappa\in j(S)$ and $j,M\in N$.
\end{enumerate}

Recall that $\kappa$ is inaccessible if and only if it is $\Pi^1_0$-indescribable, and in this case the $\Pi^1_0$-indescribable ideal equals $\NS_\kappa$. For $\kappa$ a regular cardinal, we say that a set $C\subseteq\kappa$ is \emph{$1$-closed below $\kappa$} if for all inaccessible $\alpha<\kappa$, $C\cap\alpha$ stationary in $\alpha$ implies $\alpha\in C$. Furthermore, we say that $C\subseteq\kappa$ is \emph{$1$-club} if it is stationary in $\kappa$ and $1$-closed below $\kappa$. More generally, for $n<\omega$ we say that $C\subseteq\kappa$ is \emph{$(n+1)$-closed below $\kappa$} if for all $\Pi^1_n$-indescribable $\alpha<\kappa$, $C\cap\alpha$ being $\Pi^1_n$-indescribable in $\alpha$ implies $\alpha\in C$. We say that $C\subseteq\kappa$ is \emph{$(n+1)$-club} if $C$ is $\Pi^1_n$-indescribable in $\kappa$ and $n$-closed below $\kappa$.

The following result, due to Sun \cite{MR1245524} in the case $n=1$ and due to Hellsten \cite{MR2026390} for $n>1$, provides a characterization of $\Pi^1_n$-indescribability which resembles the definition of stationarity.

\begin{theorem}[Sun, Hellsten]\label{theorem_n_club_characterization}
Suppose $\kappa$ is $\Pi^1_n$-indescribable. The following are equivalent for all $S\subseteq\kappa$.
\begin{enumerate}
\item $S$ is $\Pi^1_n$-indescribable.
\item For all $n$-clubs $C\subseteq\kappa$ we have $S\cap C\neq\emptyset$.
\end{enumerate}
\end{theorem}

An easy consequence of Theorem \ref{theorem_n_club_characterization} is that when $\kappa$ is $\Pi^1_n$-indescribable, the filter dual to the $\Pi^1_n$-indescribable ideal on $\kappa$ equals the filter generated by the collection of $n$-club subsets of $\kappa$. In other words, we have
\[\Pi^1_n(\kappa)^*=\{X\subseteq\kappa\st \text{there is an $n$-club $C\subseteq\kappa$ such that $C\subseteq X$}\}\]
Let us point out here that it is consistent that the $1$-club filter on $\kappa$ is nontrivial when $\kappa$ is not weakly compact (see \cite[Theorem 1.18]{MR1245524}).

\begin{remark}
Recall that when $\kappa$ is inaccessible and $S\subseteq\kappa$ is stationary there is a natural forcing to add a club subset of $S$, namely
\[\text{CU}(S)=\{p\st \text{$p\subseteq S$, $|p|<\kappa$ and $p$ is a closed set of ordinals}\}\]
where conditions in $\text{CU}(S)$ are ordered by letting $q\leq p$ if and only if $q$ is an end extension of $p$, meaning that $p=q\cap\sup\{\alpha+1\st\alpha\in p\}$. A generic filter for $\text{CU}(S)$ yields a club subset of $\kappa$ which is contained in $S$. In general, there is no reason to expect that this forcing preserves cofinalities. However, if $S$ contains the singular cardinals below $\kappa$, then the forcing $\text{CU}(S)$ is well-behaved: if $\SING\cap\kappa\subseteq S$ then for every regular $\gamma<\kappa$ the poset $\text{CU}(S)$ contains a dense set that is $\gamma$-closed. See \cite{MR805967} for more details.
\end{remark}

Suppose $\gamma$ is an inaccessible cardinal and $A\subseteq\gamma$ is cofinal. For $n\geq 1$, we define a poset $T^n(A)$ consisting of all bounded $n$-closed $c\subseteq A$ ordered by end extension: $c\leq d$ if and only if $d=c\cap\sup_{\alpha\in d}(\alpha+1)$. A generic filter for $T^n(A)$ will in many cases provide an $n$-club contained in $A$. Notice that these results show that $n$-club shooting forcings are, in general, more well-behaved than club shooting forcings. Indeed, let us show that $T^n(A)$ often has a closure property which club shooting forcings lack. 

Given a forcing $\P$ and an ordinal $\alpha$, we define $\mathcal{G}_\alpha(\P)$, a two-player game of perfect information as follows. Player I and Player II take turns to play conditions from $\P$ for $\alpha$ many moves, where Player I plays at odd stages and Player II plays at even stages (including all limit stages). Let $p_\beta$ be the condition played at move $\beta$. The player who played $p_\beta$ loses immediately unless $p_\beta\leq p_\gamma$ for all $\gamma<\beta$. If neither player loses at any stage $\beta<\alpha$, then Player II wins. As in \cite[Definition 5.15]{MR2768691}, we say that $\P$ is \emph{$\kappa$-strategically closed} if and only if Player II has a winning strategy for $\mathcal{G}_\kappa(\P)$.

\begin{lemma}[Hellsten \cite{MR2653962}]
For $n\geq 1$, if $\gamma$ is inaccessible and $A\subseteq\gamma$ is cofinal, then $T^n(A)$ is $\gamma$-strategically closed.
\end{lemma}

\begin{proof}
We describe a winning strategy for player II in the game $\mathcal{G}_\kappa(T^n(A))$. Player II begins the game by playing $c_0=\emptyset$. At an even successor stage $\alpha+2$, player II chooses a condition $c_{\alpha+2}\in T^n(A)$ such that $c_{\alpha+2}\lneq c_{\alpha+1}$. At limit stages $\alpha<\gamma$, player II records an ordinal $\gamma_\alpha=\bigcup_{\beta<\alpha}c_\beta$, chooses an element $\eta_\alpha\in A\setminus(\gamma_\alpha+1)$ and plays $c_\alpha=\left(\bigcup_{\beta<\alpha}c_\beta\right) \cup\{\eta_\alpha\}$. In order to argue that $c_\alpha$ is a condition in $T^n(A)$, we need to verify, letting $c=\bigcup_{\beta<\alpha}c_\beta$, that $c$ is not a $\Pi^1_{n-1}$-indescribable subset of $\gamma_\alpha$.  We can assume that
$\gamma_\alpha$ is $\Pi^1_{n-1}$-indescribable, as otherwise $c \cap \gamma_\alpha$ is clearly not $\Pi^1_{n-1}$-indescribable. But then, by construction,
$\{\gamma_\xi \st \xi < \alpha \text{ is a limit ordinal}\}$ is a club (and hence an $(n-1)$-club) in $\gamma_\alpha$ disjoint from $c$, which implies that
$c$ is not a $\Pi^1_{n-1}$-indescribable subset of $\gamma_\alpha$. Thus, $c_\alpha$ is a valid play by Player II, and we have described a winning strategy
in $\mathcal{G}_\kappa(T^n(A))$.
\end{proof}

The following result is due to Hellsten (see \cite{MR2026390} and \cite{MR2653962}) in the case in which $n=1$ and to \cite{CodyGitmanLambieHanson} for $n>1$. Notice the forcing is the same for all $n<\omega$ in the proof of the following. The difficult part for $n>1$ is arguing that $\Pi^1_n$-indescribable sets are preserved.

\begin{theorem}[Hellsten (for $n=1$), Cody, Gitman, Lambie-Hanson (for $n>1$)] \label{theorem_n_club_shooting}
Suppose that $n \geq 1$ and $S\subseteq\kappa$ is $\Pi^1_n$-indescribable. Then there is a cofinality-preserving forcing extension in which $S$
contains a $1$-club and all $\Pi^1_n$-indescribable subsets of $S$ from $V$ remain $\Pi^1_n$-indescribable.\footnote{Notice that the conclusion ``$S$ contains a $1$-club'' directly implies that ``$S$ contains an $n$-club'' because for $n<\omega$, every $n$-club is an $(n+1)$-club.}
\end{theorem}

\begin{proof}
We only provide a definition of the forcing. The reader should consult \cite{CodyGitmanLambieHanson} for details.

Let $\P_{\kappa+1}=\<(\P_\alpha,\dot{\Q}_\beta)\st\alpha\leq\kappa+1, ~ \beta\leq\kappa\>$ be an Easton-support iteration such that
\begin{itemize}
\item if $\gamma\leq\kappa$ is inaccessible and $S\cap\gamma$ is cofinal in $\gamma$, then $\dot{\Q}_\gamma=(T^1(S\cap\gamma))^{V^{\P_\gamma}}$;
\item otherwise, $\dot{\Q}_\gamma$ is a $\P_\gamma$-name for trivial forcing.
\end{itemize}
\end{proof}

\subsection{A theorem of Hellsten}

Recall that the Mahlo operation $M:P(\kappa)\to P(\kappa)$ defined by $M(X)=\{\alpha<\kappa\st \text{$X\cap\alpha$ is stationary in $\alpha$}\}$, can be iterated as follows to define the hierarchy of $\alpha$-Mahlo cardinals for $\alpha<\kappa^+$. Baumgartner, Taylor and Wagon \cite{MR687276} defined $M^\alpha:P(\kappa)/\NS_\kappa\to P(\kappa)/\NS_\kappa$ for $\alpha<\kappa^+$ by letting $M^0([X])=[X]$, $M([X])=[M(X)]$, $M^{\alpha+1}([X])=M(M^\alpha([X])$ and $M^\beta([X])=\bigtriangleup\{M^\alpha([X])\st\alpha<\beta\}$ where $\beta<\kappa^+$ is a limit ordinal and the diagonal intersection is taken relative to some function $f:\kappa\to\{M^\alpha([X])\st\alpha<\beta\}$. It is well known that such diagonal intersections are independent of the indexing used when the ideal one is working with is normal. A cardinal $\kappa$ is called \emph{$\alpha$-Mahlo}, where $\alpha\leq\kappa^+$, if for all $\beta<\alpha$ we have $M^\beta([\kappa])>0$. A cardinal $\kappa$ is called \emph{greatly Mahlo} if it is $\kappa^+$-Mahlo. 

If $\kappa$ is $\kappa^+$-Mahlo then there is a natural decomposition of $\kappa$ into $\kappa^+$ almost disjoint stationary sets \cite{MR687276}, and hence 
\[\NS_\kappa\restrict\REG=\{X\subseteq\kappa\st X\cap\REG\in\NS_\kappa\}\] 
is not $\kappa^+$-saturated. Jech and Woodin \cite{MR805967} showed that this result is sharp in the sense that $\NS_\kappa\restrict\REG$ can be saturated when $\kappa$ is $\alpha$-Mahlo where $\alpha<\kappa^+$, relative to the existence of a measurable cardinal of Mitchell order $\alpha$. In fact, Jech and Woodin gave an equiconsistency. If $\kappa$ is $\alpha$-Mahlo where $\alpha<\kappa^+$ and $\NS_\kappa\restrict\REG$ is $\kappa^+$-saturated, then there is an inner model with a measurable cardinal of Mitchell order $\alpha$. Furthermore, if $\kappa$ is measurable with Mitchell order $\alpha<\kappa^+$ then there is a forcing extension in which $\kappa$ is $\alpha$-Mahlo and $\NS_\kappa\restrict\REG$ is $\kappa^+$-saturated. By Gitik-Shelah \cite{MR1363421}, restricting to the regulars is necessary.

Hellsten showed that the forcing construction of Jech and Woodin mentioned above can be generalized from the nonstationary ideal $\NS_\kappa$ to the the weakly compact ideal $\Pi^1_1(\kappa)$. First, Hellsten iterated the generalized Mahlo operation $\Tr_1:P(\kappa)\to P(\kappa)$ defined by
\[\Tr_1(X)=\{\alpha<\kappa\st \text{$X\cap\alpha$ is weakly compact in $\alpha$}\}\]
modulo the weakly compact ideal in order to define the notion of $\kappa^+$--weak compactness. For $A\subseteq\kappa$ we let $[A]_1$ denote the equivalence class of $A$ modulo the weakly compact ideal and, as in the definition of $\alpha$-Mahlo cardinals and the nonstationary ideal, working modulo the weakly compact ideal we can iterate $\Tr_1$ to define a sequence $\<\Tr^\alpha_1([A]_1)\st \alpha<\kappa^+\>$ in the boolean algebra $P(\kappa)/\Pi^1_1(\kappa)$.

\begin{definition}
A cardinal $\kappa$ is \emph{$\alpha$-weakly compact}, where $\alpha\leq\kappa^+$, if and only if $\Tr_1^\beta([\kappa]_1)>0$ for all $\beta<\alpha$.
\end{definition}

As in the case of $\kappa^+$-Mahloness and the nonstationary ideal, it is easy to see that if $\kappa$ is $\kappa^+$--weakly compact then the weakly compact ideal $\Pi^1_1(\kappa)$ is not $\kappa^+$-saturated. Hellsten used an Easton support iteration of iterated $1$-club shootings to prove that, relative to the existence of a measurable cardinal, the weakly compact ideal $\Pi^1_1(\kappa)$ can be $\kappa^+$-saturated when $\kappa$ is weakly compact (take $\alpha=1$ in the following).

\begin{theorem}[Hellsten \cite{MR2653962}]
If $\kappa$ is measurable with Mitchell order $\alpha\in\kappa^+\setminus\{0\}$ then there is a forcing extension in which $\kappa$ is $\alpha$-$\Pi^1_1$-indescribable and the weakly compact ideal $\Pi^1_1(\kappa)$ is $\kappa^+$-saturated.
\end{theorem}

\subsection{The weakly compact reflection principle}\label{section_weakly_compact_reflection}

The \emph{weakly compact reflection principle} holds at $\kappa$, which we write as $\Refl_1(\kappa)$ or $\Refl_{wc}(\kappa)$, if and only if $\kappa$ is a weakly compact cardinal and every weakly compact subset of $\kappa$ has a weakly compact proper initial segment; in other words, $\kappa$ is weakly compact and $S\in\Pi^1_1(\kappa)^+$ implies there is an $\alpha<\kappa$ such that $S\cap\alpha\in\Pi^1_1(\alpha)^+$. It is easy to see that the weakly compact reflection principle $\Refl_1(\kappa)$ implies that $\kappa$ is $\omega$-weakly compact (see \cite{MR3985624}). In this section we address the question: what is the relationship between the weakly compact reflection principle $\Refl_1(\kappa)$ and the $\alpha$--weak compactness of $\kappa$? Several related questions remain open.

 It follows from a result of Bagaria, Magidor and Sakai \cite{MR3416912} that, in $L$, a cardinal $\kappa$ is $\Pi^1_2$-indescribable if and only if $\Refl_1(\kappa)$. Since there are many $\omega$-weakly compact cardinals below any $\Pi^1_2$-indescribable cardinal, the Bagaria-Magidor-Sakai result shows that $\kappa$ being $\omega$-weakly compact does not imply $\Refl_1(\kappa)$. The author \cite{MR3985624} gave a forcing construction for adding a non-reflecting weakly compact set which also establishes that the $\alpha$-weak compactness of $\kappa$, where $\alpha<\kappa^+$, does not imply $\Refl_1(\kappa)$.

\begin{theorem}[Cody \cite{MR3985624}]
Suppose $\kappa$ is $(\alpha+1)$-weakly compact where $\alpha<\kappa^+$. Then there is a cofinality-preserving forcing extension in which $\kappa$ remains $(\alpha+1)$-weakly compact and there is a weakly compact subset of $\kappa$ with no weakly compact proper initial segment, thus $\Refl_1(\kappa)$ fails in the extension.
\end{theorem}

One may also wonder, does the weakly compact reflection principle $\Refl_1(\kappa)$ imply that $\kappa$ must be $(\omega+1)$-weakly compact? The author and Sakai \cite{MR4050036} proved that the answer is no.

\begin{theorem}[Cody-Sakai \cite{MR4050036}]\label{theorem_cody_sakai}
Suppose the weakly compact reflection principle $\Refl_1(\kappa)$ holds. Then there is a forcing extension in which $\Refl_1(\kappa)$ holds and $\kappa$ is the least $\omega$-weakly compact cardinal.
\end{theorem}

In the course of proving Theorem \ref{theorem_cody_sakai}, the authors generalized the well-known fact that after $\kappa$-c.c. forcing the nonstationary ideal of the extension equals the ideal generated by the ground model nonstationary ideal, to the weakly compact ideal.

\begin{theorem}[Cody-Sakai \cite{MR4050036}]\label{theorem_weakly_compact_ideal_after_forcing}
Suppose $\kappa$ is a weakly compact cardinal, assume that $\P=\<(\P_\alpha,\dot{\Q}_\alpha)\st\alpha<\kappa\>\subseteq V_\kappa$ is a good\footnote{Such an iteration is \emph{good} if if for all
$\alpha < \kappa$, if $\alpha$ is inaccessible, then $\dot{\Q}_\alpha$ is a $\P_\alpha$-name for a poset such that $\forces_{\P_\alpha} \dot \Q_\alpha\in \dot V_\kappa$,
where $\dot{V}_\kappa$ is a $\P_\alpha$-name for $(V_\kappa)^{V^{\P_\alpha}}$ and, otherwise, $\dot{\Q}_\alpha$ is a $\P_\alpha$-name for trivial forcing.} Easton-support forcing iteration such that for each $\alpha<\kappa$, $\forces_{\P_\alpha}$ ``$\dot{\Q}_\alpha$ is $\alpha$-strategically closed''. Let $\dot{\Pi}^1_1(\kappa)$ be a $\P_\kappa$-name for the weakly compact ideal of the extension $V^{\P_\kappa}$ and let $\check{\Pi}^1_1(\kappa)$ be a $\P_\kappa$-check name for the weakly compact ideal of the ground model. Then $\forces_{\P_\kappa}$ $\dot{\Pi}^1_1(\kappa)=\overline{\check{\Pi}^1_1(\kappa)}$. 
\end{theorem}

\begin{proof}
Since ground-model $1$-club subsets of $\kappa$ remain $1$-club after $\kappa$-c.c. forcing, it is easy to see that $\forces_{\P_\kappa}$ $\dot{\Pi}^1_1(\kappa)\supseteq \overline{\check{\Pi}^1_1(\kappa)}$.

To show that $\forces_{\P_\kappa}$ $\dot{\Pi}^1_1(\kappa)\subseteq\overline{\check{\Pi}^1_1(\kappa)}$, we will show that if $p\forces_{\P_\kappa}$ $\dot{X}\in \dot{\Pi}^1_1(\kappa)$, then there is $B\in\Pi^1_1(\kappa)^V$ with $p\forces_{\P_\kappa}$ $\dot{X}\subseteq B$. Suppose $p\forces_{\P_\kappa}$ $\dot{X}\in \dot{\Pi}^1_1(\kappa)$. Since $\P_\kappa\subseteq V_\kappa$ is $\kappa$-c.c. we may assume that $\dot{X}\in H_{\kappa^+}$. By the fullness principle, take a $\P_\kappa$-name $\dot{A}\in H_{\kappa^+}$ for a subset of $\kappa$ such that
\begin{align*}
p\forces_{\P_\kappa} & \text{``for every $\kappa$-model $M$ with $\kappa,\dot{A},\dot{X}\in M$ and for every}\tag{$*$}\\
 & \text{elementary embedding $j:M\to N$ where $N$ is a $\kappa$-model}\\
 & \text{we have $\kappa\notin j(\dot{X})$''}
\end{align*}
Let $B=\{\alpha<\kappa\st\exists q\in\P_\kappa\ (q\leq p)\land (q\forces_{\P_\kappa} \alpha\in \dot{X})\}$ and notice that $B\in V$ and $p\forces_{\P_\kappa}$ $\dot{X}\subseteq B$. Thus, to complete the proof it will suffice to show that $B\in\Pi^1_1(\kappa)^V$.

Suppose $B\notin\Pi^1_1(\kappa)^V$. Using the weak compactness of $B$ in $V$, let $M$ be a $\kappa$-model with $\kappa,B,\dot{A},\P_\kappa,\dot{X},p,\ldots\in M$ and let $j:M\to N$ be an elementary embedding with critical point $\kappa$ such that $\kappa\in j(B)$ where $N$ is a $\kappa$-model. Since $\kappa\in j(B)$, it follows by elementarity that there is a condition $r\in j(\P_\kappa)$ with $r\leq j(p)=p$ such that $r\forces_{j(\P_\kappa)}\kappa\in j(\dot{X})$. Let $G\subseteq\P_\kappa$ be generic over $V$ with $r\restrict\kappa\in G$. Since $\P_\kappa$ is $\kappa$-c.c. the model $N[G]$ is closed under ${<}\kappa$-sequences in $V[G]$. Furthermore, the poset $j(\P_\kappa)/G$ is $\kappa$-strategically closed in $N[G]$. Thus, working in $V[G]$ we can build a filter $H\subseteq j(\P_\kappa)/G$ which is generic over $N[G]$ with $r/G\in H$. Let $\hat{G}$ denote the filter for $j(\P_\kappa)$ obtained from $G*H$ and notice that $r\in \hat{G}$. Since conditions in $\P_\kappa$ have support bounded below the critical point of $j$, it follows that $j[G]\subseteq \hat{G}$. Thus the embedding extends to $j:M[G]\to N[\hat{G}]$. Since $r\in \hat{G}$ and $r\forces_{j(\P_\kappa)}$ $\kappa\in j(\dot{X})$, we have $\kappa\in j(\dot{X}_G)$. Notice that $p\in G$; this contradicts ($*$) since $M[G]$ and $N[\hat{G}]$ are $\kappa$-models.
\end{proof}

See Section \ref{section_questions} for some relevant open questions.

\subsection{A $\square(\kappa)$-like principle consistent with weak compactness}

Recall that Todor\v{c}evi\'{c}'s principle $\square(\kappa)$ \cite{MR908147} implies that $\kappa$ is not weakly compact. In this section we survey some results from \cite{CodyGitmanLambieHanson}, which concern forcing a $\square(\kappa)$-like principle to hold at a weakly compact cardinal $\kappa$.

Recall that Todor\v{c}evi\'{c}'s principle $\square(\kappa)$ asserts that there is a $\kappa$-length coherent sequence of clubs $\vec{C}=\<C_\alpha\st\alpha\in\lim(\kappa)\>$ that cannot be threaded. For an uncountable cardinal $\kappa$, a sequence $\vec{C}=\<C_\alpha\st\alpha\in\lim(\kappa)\>$ of clubs $C_\alpha\subseteq\alpha$ is called \emph{coherent} if whenever $\beta$ is a limit point of $C_\alpha$ we have $C_\beta=C_\alpha\cap\beta$. Given a coherent sequence $\vec{C}$, we say that $C$ is a \emph{thread} through $\vec{C}$ if $C$ is a club subset of $\kappa$ and $C\cap\alpha=C_\alpha$ for every limit point $\alpha$ of $C$. A coherent sequence $\vec{C}$ is called a \emph{$\square(\kappa)$-sequence} if it cannot be threaded. Gitman, Lambie-Hanson and the author \cite{CodyGitmanLambieHanson}, and independently Welch and Brickhill \cite{BrickhillWelch}, introduced generalizations of $\square(\kappa)$, by replacing the use of clubs with $n$-clubs. The Brickhill-Welch principle $\square^n(\kappa)$ is defined differently from the principle $\square_n(\kappa)$ from \cite{CodyGitmanLambieHanson} which we present here. Although it is easy to see that when $\kappa$ is $\Pi^1_n$-indescribable the Brickhill-Welch principle $\square^n(\kappa)$ implies $\square_n(\kappa)$, but it is not known whether the converse holds. See \cite{CodyGitmanLambieHanson} for more information on the relationship between $\square^n(\kappa)$ and $\square_n(\kappa)$.

For $n<\omega$ and $X\subseteq\kappa$, we define the \emph{$n$-trace of $X$} to be
\[\Tr_n(X)=\{\alpha<\kappa\st X\cap\alpha\in \Pi^1_n(\alpha)^+\}.\]
Notice that when $X=\kappa$, $\Tr_n(\kappa)$ is the set of $\Pi^1_n$-indescribable cardinals below $\kappa$, and in particular $\Tr_0(\kappa)$ is the set of inaccessible cardinals less than $\kappa$.
For uniformity of notation, let us say that an ordinal $\alpha$ is \emph{$\Pi^1_{-1}$-indescribable} if it is a limit ordinal, and if $\alpha$ is a limit ordinal, $S \subseteq \alpha$ is
\emph{$\Pi^1_{-1}$-indescribable} if it is unbounded in $\alpha$. Thus, if $X \subseteq \kappa$, then 
\[\Tr_{-1}(\kappa) = \{\alpha < \kappa \st \text{$\alpha$ is a limit ordinal and $\sup(X \cap \alpha) = \alpha$}\}.\]

\begin{definition} \label{definition_n_square}
Suppose $n<\omega$ and $\Tr_{n-1}(\kappa)$ is cofinal in $\kappa$. A sequence $\vec{C}=\<C_\alpha\st\alpha\in \Tr_{n-1}(\kappa)\>$ is called a \emph{coherent sequence of $n$-clubs} if
\begin{enumerate}
\item for all $\alpha\in \Tr_{n-1}(\kappa)$, $C_\alpha$ is an $n$-club subset of $\alpha$ and
\item for all $\alpha<\beta$ in $\Tr_{n-1}(\kappa)$, $C_\beta\cap\alpha\in\Pi^1_{n-1}(\alpha)^+$ implies $C_\alpha=C_\beta\cap\alpha$.
\end{enumerate}
We say that a set $C\subseteq\kappa$ is a \emph{thread} through a coherent sequence of $n$-clubs $$\vec{C}=\<C_\alpha\st\alpha\in \Tr_{n-1}(\kappa)\>$$ if $C$ is $n$-club and for all $\alpha\in \Tr_{n-1}(\kappa)$, $C\cap\alpha\in\Pi^1_{n-1}(\alpha)^+$ implies $C_\alpha=C\cap\alpha$.
A coherent sequence of $n$-clubs $\vec{C}=\<C_\alpha\st\alpha\in \Tr_{n-1}(\kappa)\>$ is called a \emph{$\square_n(\kappa)$-sequence} if there is no thread through $\vec{C}$. We say that $\square_n(\kappa)$ holds if there is a $\square_n(\kappa)$-sequence $\vec{C}=\<C_\alpha\st\alpha\in\Tr_{n-1}(\kappa)\>$.
\end{definition}

\begin{remark}
Note that $\square_0(\kappa)$ is simply $\square(\kappa)$. For $n=1$, the principle $\square_1(\kappa)$ states that there is a coherent sequence of $1$-clubs \[\< C_\alpha\st\alpha < \kappa \text{ is inaccessible}\>\] that cannot be threaded.
\end{remark}

Generalizing the fact that $\square(\kappa)$ implies $\kappa$ is not weakly compact, we prove the following fact from \cite{CodyGitmanLambieHanson}.

\begin{proposition}\label{proposition_n_square_denies_n_plus_1_indescribability}
For every $n<\omega$, $\square_n(\kappa)$ implies that $\kappa$ is not $\Pi^1_{n+1}$-indescriba\-ble.
\end{proposition}

\begin{proof}
Suppose $\vec{C}=\<C_\alpha\st\alpha\in\Tr_n(\kappa)\>$ is a $\square_n(\kappa)$-sequence and $\kappa$ is $\Pi^1_{n+1}$-indescr\-ibable. Let $M$ be a $\kappa$-model with $\vec{C}\in M$. Since $\kappa$ is $\Pi^1_{n+1}$-indescribable, we may let $j:M\to N$ be an elementary embedding with critical point $\kappa$ and a $\Pi^1_n$-correct $N$ as in Theorem \ref{theorem_hauser}(2). By elementarity, it follows that $j(\vec{C})=\<\bar{C}_\alpha\st\alpha\in\Tr_{n-1}^N(j(\kappa))\>$ is a $\square_n(j(\kappa))$-sequence in $N$. Since $N$ is $\Pi^1_n$-correct, we know that $\kappa \in \Tr_{n-1}^N(j(\kappa))$ and $\bar C_\kappa$ must also be $n$-club in $V$. Since $j(\vec{C})$ is a $\square_n(j(\kappa))$-sequence in $N$, it follows that for every $\Pi^1_{n-1}$-indescribable $\alpha<\kappa$ if $\bar{C}_\kappa\cap\alpha\in\Pi^1_{n-1}(\alpha)^+$, then $\bar{C}_\kappa\cap\alpha=C_\alpha$, and hence $\bar{C}_\kappa$ is a thread through $\vec{C}$, a contradiction.
\end{proof}

\begin{remark}
It follows easily that $\square_1(\kappa)$ holds trivially at weakly compact cardinals $\kappa$ below which stationary reflection fails (see \cite{CodyGitmanLambieHanson}). Thus, the task at hand  is not just to force $\square_1(\kappa)$ to hold while $\kappa$ is weakly compact, but to force $\square_1(\kappa)$ to hold while $\kappa$ is weakly compact and stationary reflection holds often below $\kappa$, so that the coherence requirement in $\square_1(\kappa)$ is nontrivial.
\end{remark}

By adapting techniques used to force the existence of $\square(\kappa)$-sequences, Gitman, Lambie-Hanson and the author \cite{CodyGitmanLambieHanson} proved the following two results.

\begin{theorem}[Cody, Gitman and Lambie-Hanson \cite{CodyGitmanLambieHanson}]\label{theorem_1_square_wc}
If $\kappa$ is $\kappa^+$-weakly compact and the $\GCH$ holds, then there is a cofinality-preserving forcing extension in which
\begin{enumerate}
\item $\kappa$ remains $\kappa^+$-weakly compact and
\item $\square_1(\kappa)$ holds.
\end{enumerate}
\end{theorem}

\begin{theorem}[Cody, Gitman and Lambie-Hanson \cite{CodyGitmanLambieHanson}]\label{theorem_1_square_and_reflection}
Suppose that $\kappa$ is $\Pi^1_2$-indescribable and the $\GCH$ holds. Then there is a cofinality-preserving forcing extension in which
\begin{enumerate}
\item $\square_1(\kappa)$ holds,
\item $\Refl_1(\kappa)$ holds and
\item $\kappa$ is $\kappa^+$-weakly compact.
\end{enumerate}
\end{theorem}

Additionally, one can prove \cite{CodyGitmanLambieHanson}, that $\square_n(\kappa)$ is incompatible with simultaneous reflection of $\Pi^1_n$-indescribable sets; a similar result was proven by Brickhill and Welch \cite{BrickhillWelch} using their principle $\square^n(\kappa)$.

\begin{theorem}\label{theorem_n_square_refutes_simultaneous_refl}
  Suppose that $1\leq n<\omega$, $\kappa$ is $\Pi^1_n$-indescribable and $\square_n(\kappa)$ holds. Then
  there are two $\Pi^1_n$-indescribable subsets $S_0, S_1 \subseteq \kappa$ that do not
  reflect simultaneously, i.e., there is no $\beta < \kappa$ such that
  $S_0 \cap \beta$ and $S_1 \cap \beta$ are both $\Pi^1_n$-indescribable subsets of $\beta$.
\end{theorem}

\section{Large cardinal ideals on $P_\kappa\lambda$}\label{section_p_kappa_lambda}








In this section we survey some results on two-cardinal versions of subtlety, ineffability and indescribability. Recall that, for cardinals $\kappa\leq\lambda$, Jech defined notions of closed unbounded and stationary subsets of $P_\kappa\lambda$. Carr \cite{MR667297} proved that, given a regular cardinal $\kappa$, Jech's nonstationary ideal $\NS_{\kappa,\lambda}$ is the minimal normal fine $\kappa$-complete ideal on $P_\kappa\lambda$. Using a two-cardinal version of the cumulative hierarchy up to $\kappa$, Baumgartner introduced a notion of $\Pi^1_n$-indescribability for sets $S\subseteq P_\kappa\lambda$ (see Section \ref{section_two_cardinal_indescribability} below), which led to the consideration of associated ideals
\[\Pi^1_n(\kappa,\lambda)=\{X\subseteq P_\kappa\lambda\st\text{$X$ is not $\Pi^1_n$-indescribable}\}.\]
Recall that a cardinal $\kappa$ is inaccessible if and only if it is $\Pi^1_0$-indescribable, and furthermore, when $\kappa$ is inaccessible we have $\Pi^1_0(\kappa)=\NS_\kappa$. Thus, one might expect that when $P_\kappa\lambda$ is $\Pi^1_0$-indescribable the $\Pi^1_0$-indescribable ideal $\Pi^1_0(\kappa,\lambda)$ equals Jech's nonstationary ideal on $P_\kappa\lambda$. However, this is not the case. The author proved \cite{MR4082998} that, when it is nontrivial, the ideal $\Pi^1_0(\kappa,\lambda)$ equals the minimal \emph{strongly normal} ideal on $P_\kappa\lambda$, which does not equal $\NS_\kappa$, and which consists of all non--\emph{strongly stationary} subsets of $P_\kappa\lambda$ (see Section \ref{section_strong_stationarity} for a discussion of strong stationarity and Section \ref{section_two_cardinal_indescribability} for more details regarding this result). This leads to a definition of \emph{$1$-club subset of $P_\kappa\lambda$} (see Definition \ref{definition_1_club} below), and the following characterization of $\Pi^1_1$-indescribable subsets of $P_\kappa\lambda$: when $P_\kappa\lambda$ is $\Pi^1_1$-indescribable, a set $S\subseteq P_\kappa\lambda$ is $\Pi^1_1$-indescribable if and only if $S\cap C\neq\emptyset$ for all $1$-clubs $C\subseteq P_\kappa\lambda$ (see Section \ref{section_two_cardinal_indescribability} below).

In Section \ref{section_strongly_subtle} we survey some results on two-cardinal versions of subtlety and ineffability.

Finally, in Section \ref{section_two_cardinal_generic_embeddings}, motivated by \cite{MR3913154}, we give generic embedding characterizations of two-cardinal indescribability, subtlety and ineffability.

\subsection{Stationary vs. strongly stationary subsets of $P_\kappa\lambda$}\label{section_strong_stationarity}


Throughout this section we assume $\kappa\leq\lambda$ are cardinals and $\kappa$ is a regular cardinal. Recall that an ideal $I$ on $P_\kappa\lambda$ is \emph{normal} if for every $X\in I^+$ and every function $f:P_\kappa\lambda\to\lambda$ with $\{x\in X \st f(x)\in x\}\in I^+$ there is a $Y\in P(X)\cap I^+$ such that $f\restrict Y$ is constant. Equivalently, an ideal $I$ on $P_\kappa\lambda$ is normal if and only if for every $\{X_\alpha\st\alpha<\kappa\}\subseteq I$ the set $\diagonalunion_{\alpha<\kappa}X_\alpha=_{\defn}\{x\st\text{$x\in X_\alpha$ for some $\alpha\in x$}\}$ is in $I$. An ideal $I$ on $P_\kappa\lambda$ is \emph{fine} if and only if $\{x\in P_\kappa\lambda\st \alpha\in x\}\in I^*$ for every $\alpha<\lambda$. Jech \cite{MR0325397} generalized the notion of closed unbounded and stationary subsets of cardinals to subsets $P_\kappa\lambda$ and proved that the nonstationary ideal $\NS_{\kappa\lambda}$ is a normal fine $\kappa$-complete ideal on $P_\kappa\lambda$. Carr \cite{MR667297} proved that, when $\kappa$ is a regular cardinal, the nonstationary ideal $\NS_{\kappa,\lambda}$ is the minimal normal fine $\kappa$-complete ideal on $P_\kappa\lambda$.


When considering ideals on $P_\kappa\lambda$ for $\kappa$ inaccessible, it is quite fruitful to work with a different notion of closed unboundedness obtained by replacing the structure $(P_\kappa\lambda,\subseteq)$ with a different one. For $x\in P_\kappa\lambda$ we define $\kappa_x=|x\cap\kappa|$ and we define an ordering $(P_\kappa\lambda,<)$ by letting 
\[\text{$x<y$ if and only if $x\in P_{\kappa_y}y$}.\] Given a function $f:P_\kappa\lambda\to P_\kappa\lambda$ we let
\[C_f=_{\defn}\{x\in P_\kappa\lambda\st x\cap\kappa\neq\emptyset\land f[P_{\kappa_x}x]\subseteq P_{\kappa_x}x\}.\] 
We say that a set $C\subseteq P_\kappa\lambda$ is \emph{weakly closed unbounded} if there is an $f$ such that $C= C_f$. Moreover, $X\subseteq P_\kappa\lambda$ is called \emph{strongly stationary} if for every $f$ we have $C_f\cap X\neq\emptyset$. An ideal $I$ on $P_\kappa\lambda$ is \emph{strongly normal} if for any $X\in I^+$ and function $f:P_\kappa\lambda\to P_\kappa\lambda$ such that $f(x)< x$ for all $x\in X$ there is $Y\in P(X)\cap I^+$ such that $f\restrict Y$ is constant. It follows easily that an ideal $I$ on $P_\kappa\lambda$ is strongly normal if and only if for any $\{X_a\st a\in P_\kappa\lambda\}\subseteq I$ the set $\diagonalunion_{<} X_a=_{\defn}\{x\st \text{$x\in X_a$ for some $a< x$}\}$ is in $I$. Note that an easy argument shows that if $\kappa$ is $\lambda$-supercompact then the prime ideal dual to a normal fine ultrafilter on $P_\kappa\lambda$ is strongly normal. Matet \cite{MR954259} showed that if $\kappa$ is Mahlo then the collection of non--strongly stationary sets
\[\NSS_{\kappa,\lambda}=_{def}\{X\subseteq P_\kappa\lambda\st\text{$\exists f:P_\kappa\lambda\to P_\kappa\lambda$ such that $X\cap C_f=\emptyset$}\}\]
is the minimal strongly normal ideal on $P_\kappa\lambda$. Improving this, Carr, Levinski and Pelletier obtained the following.

\begin{theorem}[Carr-Levinski-Pelletier \cite{MR1074449}]\label{cpl_minimal_strongly_normal_ideal} $P_\kappa\lambda$ carries a strongly normal ideal if and only if $\kappa$ is Mahlo or $\kappa=\mu^+$ where $\mu^{<\mu}=\mu$; moreover, in this case $\NSS_{\kappa,\lambda}$ is the minimal such ideal. 
\end{theorem}

In these cases, since every strongly normal ideal on $P_\kappa\lambda$ is normal, we have
\[\NS_{\kappa,\lambda}\subseteq\NSS_{\kappa,\lambda}.\]
The following lemma, due to Zwicker (see the discussion on page 61 of \cite{MR1074449}), shows that if $\kappa$ is weakly inaccessible the previous containment is strict. 
\begin{lemma}
If $\kappa$ is weakly inaccessible then $\NS_{\kappa,\lambda}$ is not strongly normal.
\end{lemma}

\begin{corollary}\label{corollary_NS_not_NSS}
If $\kappa$ is Mahlo then $\NSS_{\kappa,\lambda}$ is nontrivial and $\NS_{\kappa,\lambda}\subsetneq\NSS_{\kappa,\lambda}$.
\end{corollary}

\subsection{Indescribable subsets of $P_\kappa\lambda$}\label{section_two_cardinal_indescribability}

According to \cite{MR1635559} and \cite{MR808767}, in a set of handwritten notes, Baumgartner \cite{Baum} defined a notion of indescribability for subsets of $P_\kappa\lambda$ as follows. Give a regular cardinal $\kappa$ and a set of ordinals $A\subseteq\ORD$, consider the hierarchy:
\begin{align*}
V_0(\kappa,A)&=A\\
V_{\alpha+1}(\kappa,A)&=P_\kappa(V_\alpha(\kappa,A))\cup V_{\alpha}(\kappa,A)\\
V_\alpha(\kappa,A)&=\bigcup_{\beta<\alpha}V_\beta(\kappa,A) \text{ for $\alpha$ a limit}
\end{align*}
Clearly $V_\kappa\subseteq V_\kappa(\kappa,A)$ and if $A$ is transitive then so is $V_\alpha(\kappa,A)$ for all $\alpha\leq\kappa$. See \cite[Section 4]{MR808767} for a discussion of the restricted axioms of $\ZFC$ satisfied by $V_\kappa(\kappa,\lambda)$ when $\kappa$ is inaccessible.

\begin{definition}[Baumgartner \cite{Baum}]\label{definition_indescribable}
Let $S\subseteq P_\kappa\lambda$. We say that $S$ is \emph{$\Pi^1_n$-indescriba\-ble} if for every $R\subseteq V_\kappa(\kappa,\lambda)$ and every $\Pi^1_n$-formula $\varphi$ such that $(V_\kappa(\kappa,\lambda),\in,R)\models\varphi$, there is an $x\in S$ such that
\[\text{$x\cap\kappa=\kappa_x$ and $(V_{\kappa_x}(\kappa_x,x),\in, R\cap V_{\kappa_x}(\kappa_x,x))\models\varphi$.}\] 
\end{definition}

It is not too difficult to see that $\kappa$ is supercompact if and only if $P_\kappa\lambda$ is $\Pi^1_1$-indescribable for all $\lambda\geq\kappa$. See \cite[Section 1]{MR4082998} for a more detailed discussion of the way in which the $\Pi^1_1$-indescribability of $P_\kappa\lambda$ fits in with other large cardinal notions.

\begin{remark}
As noted in \cite{MR808767}, standard arguments using the coding apparatus available in $V_\kappa(\kappa,\lambda)$ show that if $\lambda^{<\kappa}=\lambda$, then we can replace $R\subseteq V_\kappa(\kappa,\lambda)$ in Definition \ref{definition_indescribable} by any finite sequence $R_1,\ldots,R_k$ of subsets of $V_\kappa(\kappa,\lambda)$.
\end{remark}

Abe \cite[Lemma 4.1]{MR1635559} showed that if $P_\kappa\lambda$ is $\Pi^1_n$-indescribable then 
\[\Pi^1_n(\kappa,\lambda)=\{X\subseteq P_\kappa\lambda\st \text{$X$ is not $\Pi^1_n$-indescribable}\}\]
is a strongly normal proper ideal on $P_\kappa\lambda$.

As mentioned above, the next theorem suggests that one should use strong stationarity instead of stationarity when generalizing the notion of $1$-club subset of $\kappa$ to that of $P_\kappa\lambda$.

\begin{theorem}[Cody \cite{MR4082998}]\label{lemma_Pi_1_0_ideal}
If $\kappa$ is Mahlo then $S\subseteq P_\kappa\lambda$ is in $\NSS_{\kappa,\lambda}^+$ if and only if $S$ is $\Pi^1_0$-indescribable (i.e. first-order indescribable); in other words,
\[\Pi^1_0(\kappa,\lambda)=\NSS_{\kappa,\lambda}.\]
\end{theorem}

Following \cite{MR1245524}, the author showed \cite{MR4082998} that the $\Pi^1_1$-indescribable subsets of $P_\kappa\lambda$ can be characterized using the following notion of $1$-club subsets of $P_\kappa\lambda$.

\begin{definition}[Cody \cite{MR4082998}]\label{definition_1_club}
We say that $C\subseteq P_\kappa\lambda$ is \emph{$1$-club} if and only if 
\begin{enumerate}
\item $C\in \NSS_{\kappa,\lambda}^+$ and
\item $C$ is \emph{$1$-closed}, that is, for every $x\in P_\kappa\lambda$, if $\kappa_x$ is an inaccessible cardinal and $C\cap P_{\kappa_x}x\in\NSS_{\kappa_x,x}^+$ then $x\in C$.
\end{enumerate}
\end{definition}

\begin{proposition}[Cody \cite{MR4082998}] \label{proposition_1_stationarity}
Suppose $P_\kappa\lambda$ is $\Pi^1_1$-indescribable and $S\subseteq P_\kappa\lambda$. Then $S$ is $\Pi^1_1$-indescribable if and only if $S\cap C\neq\emptyset$ for every $1$-club $C\subseteq P_\kappa\lambda$.
\end{proposition}

The author also proved \cite{MR4082998} that two-cardinal indescribability can be characterized using elementary embeddings which resemble those considered by Schanker \cite{MR2989393} in his work on \emph{nearly $\lambda$-supercompact} cardinals. We say that a set $W\subseteq P_\kappa\lambda$ is \emph{weakly compact} if for every $A\subseteq\lambda$ there is a transitive $M\models\ZFC^-$ with $\lambda,A,W\in M$ and $M^{<\kappa}\cap V\subseteq M$, a transitive $N$ and an elementary embedding $j:M\to N$ with critical point $\kappa$ such that $j(\kappa)>\lambda$ and $j"\lambda\in j(W)$. 

\begin{theorem}[Cody \cite{MR4082998}]\label{theorem_characterizations}
Suppose $\kappa\leq\lambda$ are cardinals with $\lambda^{<\kappa}=\lambda$. A set $W\subseteq P_\kappa\lambda$ is $\Pi^1_1$-indescribable if and only if it is weakly compact.
\end{theorem}

See \cite[Section 1]{MR4082998} for a detailed discussion of how the weak compactness of $P_\kappa\lambda$, or in Schanker's terminology, the near $\lambda$-supercompactness of $\kappa$ fits in with other large cardinal notions.

\subsection{Subtle, strongly subtle and ineffable subsets of $P_\kappa\lambda$} \label{section_strongly_subtle}


Menas \cite{MR0357121} defined a set $S\subseteq P_\kappa\lambda$ to be \emph{subtle} if for every club $C\subseteq P_\kappa\lambda$ and every function $\vec{S}=\<S_x\st x\in S\>$ where $S_x\subseteq x$ for all $x\in S$, there are $x,y\in C\cap S$ such that $x\subsetneq y$ and $S_x=S_y\cap x$. Menas proved that no matter how large $\lambda$ is, the subtlety of $P_\kappa\lambda$ is not a stronger assumption than the subtlety of $\kappa$.

\begin{theorem}[Menas \cite{MR0357121}]
If $\kappa$ is subtle then $P_\kappa\lambda$ is subtle for all $\lambda\geq\kappa$.
\end{theorem}

Menas also showed that the subtlety of $P_\kappa\lambda$ implies the usual two-cardinal diamond principle $\diamondsuit_{\kappa,\lambda}$, and thus the subtlety of $\kappa$ implies $\diamondsuit_{\kappa,\lambda}$ for all $\lambda\geq\kappa$. A surprising result of Usuba demonstrates that Menas's version of two-cardinal subtlety does not behave as expected: Usuba proved that, in general, the subtlety of $P_\kappa\lambda$, where $\lambda\geq\kappa$, does not imply the subtlety of $\kappa$.

\begin{theorem}[Usuba \cite{MR3472180}]\label{theorem_usuba}
Suppose $\lambda$ is a measurable cardinal. Then there is a regular uncountable $\kappa<\lambda$ such that $P_\kappa\lambda$ is subtle but $\kappa$ is not subtle.
\end{theorem}

As noted above, when studying $P_\kappa\lambda$ combinatorics where $\kappa$ is inaccessible, it can be fruitful to shift attention from the structure $(P_\kappa\lambda,\subseteq)$ to the structure $(P_\kappa\lambda,<)$ where $x<y$ if and only if $x\in P_{\kappa_y}y$ (see Section \ref{section_strong_stationarity} above). Abe \cite{MR2210149} defined a set $S\subseteq P_\kappa\lambda$ to be \emph{strongly subtle} if for every function $\vec{S}=\<S_x\st x\in S\>$ where $S_x\subseteq P_{\kappa_x}x$ for $x\in S$, and for every weak club $C\in\NSS_{\kappa,\lambda}^*$ there exist $x,y\in S\cap C$ with $x<y$ (meaning $x\in P_{\kappa_y}y$) such that $S_x=S_y\cap P_{\kappa_x}x$. Abe proved \cite{MR2210149} that the collection 
\[\text{NSSub}_{\kappa,\lambda}=\{X\subseteq P_\kappa\lambda\st \text{$X$ is not strongly subtle}\}\]
of non--strongly subtle subsets of $P_\kappa\lambda$ is a strongly normal ideal on $P_\kappa\lambda$ when it is nontrivial.
Although Abe's definition of strong subtlety of $P_\kappa\lambda$  for $\lambda>\kappa$ still does not provide a hypothesis stronger than the subtlety of $\kappa$, it does not possess the same unexpected behavior as Menas's notion of subtlety.

\begin{theorem}[Abe \cite{MR2210149}]
If $\kappa$ is subtle then $P_\kappa\lambda$ is subtle for all $\lambda\geq\kappa$. Furthermore, if $P_\kappa\lambda$ is strongly subtle for some $\lambda\geq\kappa$ then $\kappa$ is subtle.
\end{theorem}

Jech \cite{MR0325397} defined a natural two-cardinal version of ineffability, which was used by Magidor \cite{MR327518} to characterize supercompactness. Carr \cite{MR854519} showed that there is an associated normal ideal. A set $S\subseteq P_\kappa\lambda$ is said to be \emph{ineffable} if for every function $\vec{S}=\<S_x\st x\in S\>$ where $S_x\subseteq x$ for $x\in S$, there is a $D\subseteq\lambda$ such that $\{x\in S\st S_x=D\cap x\}\in\NS_{\kappa,\lambda}^+$. Carr proved $P_\kappa\lambda$ is ineffable if and only if the collection
\[\text{NIn}_{\kappa,\lambda}=\{X\subseteq P_\kappa\lambda\st\text{$X$ is not ineffable}\}\]
is a normal ideal on $P_\kappa\lambda$.


\subsection{Generic embedding characterizations of large cardinal ideals on $P_\kappa\lambda$}\label{section_two_cardinal_generic_embeddings}

In this section we provide generic embedding characterizations of various large cardinal ideals on $P_\kappa\lambda$. First, let us consider the following well-known result (see \cite{MR2768692}).

\begin{lemma}[Folklore]\label{lemma_normal_generic_ultrapower}
Suppose $\kappa$ is regular, $\kappa\leq\lambda$ with $\lambda^{<\kappa}=\lambda$, and $I$ is a $\kappa$-complete normal fine ideal on $P_\kappa\lambda$. If $G\subseteq P(P_\kappa\lambda)/I$ is generic and $j:V\to M=V^{P_\kappa\lambda}/G\subseteq V[G]$ is the corresponding generic ultrapower then the following conditions hold.
\begin{enumerate}
\item $G$ extends the filter $I^*$ dual to $I$.
\item $\crit(j)=\kappa$ and $j(\kappa)>\lambda$.
\item $[\id]_G=j"\lambda\in M$ and thus for all $X\in P(P_\kappa\lambda)^V$ we have $X\in G$ if and only if $j"\lambda\in^M j(X)$.
\item For every function $f:P_\kappa\lambda\to V$ in $V$ we have $j(f)(j"\lambda)=[f]_G$.
\item $M$ is wellfounded up to $(\lambda^+)^V$.
\end{enumerate}
\end{lemma}

This previous lemma easily leads to a generic embedding characterization of stationary subsets of $P_\kappa\lambda$.

\begin{proposition}[Folklore]
A set $S\subseteq P_\kappa\lambda$ is stationary if and only if there is a generic elementary embedding $j:V\to M\subseteq V[G]$ with critical point $\kappa$ such that $j(\kappa)>\lambda$ and $j"\lambda\in j(S)\cap M$.
\end{proposition}

\begin{proof}
Suppose $S$ is stationary and let $G\subseteq P(\kappa)/(\NS_{\kappa,\lambda}\restrict S)$ be generic. Since $\NS_{\kappa,\lambda}\restrict S$ is a $\kappa$-complete normal ideal on $P_\kappa\lambda$ we have $\crit(j)=\kappa$ and $[\id]_G=j"\lambda$. Hence $j"\lambda\in M$. Since $\lambda=j(f)(j"\lambda)=[f]_G$ where $f(x)=\ot(x)$ and $j(\kappa)=[c_\kappa]_G$ we have $\lambda<j(\kappa)$. Clearly $S\in (\NS_{\kappa,\lambda}\restrict S)^*\subseteq G$ and hence $\kappa\in j(S)$.

Conversely, suppose we have such a $j$ added by forcing with $\P$. Fix a club $C\subseteq P_\kappa\lambda$ in $V$. Then
\[I=\{X\in P(P_\kappa\lambda)\st\ \forces_\P j"\lambda\notin j(X)\}\]
is a normal ideal in $V$ and since $\NS_{\kappa,\lambda}$ is the minimal $\kappa$-complete normal ideal on $P_\kappa\lambda$ we have $\NS_{\kappa,\lambda}^*\subseteq I^*$. This implies that $C\in I^*$, in other words, $\forces_\P$ $j"\lambda\in j(C)$. Thus $M\models j(S)\cap j(C)\not=\emptyset$ and by elementarity $S\cap C\not=\emptyset$.
\end{proof}

We close by stating generic embedding characterizations of certain two-cardin\-al versions of indescribability, subtlety and ineffability. The proofs are similar to those in Section \ref{section_embeddings} above. Proposition \ref{proposition_two_cardinal_ineffability_embedding} below should be compared with \cite[Lemma 5.5]{MR3913154}.

\begin{proposition} For $n,m<\omega$, $\kappa$ regular, $\lambda\geq\kappa$ and $S\subseteq P_\kappa\lambda$, the following are equivalent.
\begin{enumerate}
\item $S$ is $\Pi^m_n$-indescribable.
\item There is a generic embedding $j:V\to M\subseteq V[G]$ with critical point $\kappa$ such that $\crit(j)=\kappa$, $j(\kappa)>\lambda$, $j"\lambda\in j(S)\cap M$ and for every $\Pi^m_n$-sentence $\varphi$ over $(V_\kappa(\kappa,\lambda),\in,A)$ where $A\in P(V_\kappa(\kappa,\lambda))^V$ we have
\[((V_\kappa(\kappa,\lambda),\in,A)\models\varphi)^V\implies((V_\kappa(\kappa,j"\lambda),\in, j(A)\cap V_\kappa(\kappa,j"\lambda))\models\varphi)^M.\]
\end{enumerate}
\end{proposition}

\begin{proposition}
A set $S\subseteq P_\kappa\lambda$ is subtle if and only if there is a generic elementary embedding $j:V\to M\subseteq V[G]$ with critical point $\kappa$ such that $j(\kappa)>\lambda$, $j"\lambda\in j(S)\cap M$ and for every function $\vec{S}=\<S_x\st x\in S\>$ where $S_x\subseteq x$ for $x\in S$, and every club $C\subseteq P_\kappa\lambda$, we have $j(\vec{S})(x)=j(\vec{d})(j"\lambda)\cap x$ for some $x\in j(S\cap C)$ with $x\subsetneq j"\lambda$.
\end{proposition}

\begin{proposition}\label{proposition_two_cardinal_ineffability_embedding}
A set $S\subseteq P_\kappa\lambda$ is ineffable if and only if for every function $\vec{S}=\<S_x\st x\in S\>$ where $S_x\subseteq x$ for $x\in S$, there is a generic elementary embedding $j:V\to M\subseteq V[G]$ with critical point $\kappa$ such that $j(\kappa)>\lambda$, $j"\lambda\in j(S)\cap M$ and $j(\vec{S})(j"\lambda)=j"D$ for some set $D\in P(\lambda)^V$.
\end{proposition}

\section{Questions}\label{section_questions}

In this section we state several open questions relating to the topics of this article. Before stating each question we refer the reader to the relevant section above for more background information and motivation.

See Section \ref{section_ramsey} and Remark \ref{remark_preramsey_question} for background concerning the first two questions.

\begin{question}\label{question_pi1n_ramsey_ideal}
For $n>1$, if $\kappa\in\R(\Pi^1_n(\kappa))^+$, does it follow that
\[\R(\Pi^1_n(\kappa))=\overline{\NPreRam_\kappa\cup\Pi^1_{n+2}(\kappa)}?\]
\end{question}

\noindent It seems that in order to answer Question \ref{question_pi1n_ramsey_ideal} one would need an answer to the following.

\begin{question}\label{question_generalize_subtle_result}
Can Baumgartner's Lemma \ref{lemma_subtlety_is_stronger_than_indescribability} be generalized to the pre-Ramsey ideal? Specifically, if $S\subseteq\kappa$ is pre-Ramsey and $f:[S]^{<\omega}\to\kappa$ is a regressive function, does it follow that $S\setminus A$ is not pre-Ramsey where
\begin{align*}
A=\{\alpha\in S\st \exists X\subseteq S\cap\alpha & \text{ such that $X$ is $\Pi^1_n$-indescribable in $\alpha$}\\
	&\text{for all $n<\omega$ and $X$ is homogeneous for $f$}\}?
\end{align*}
Baumgartner actually proved a version of Lemma \ref{lemma_subtlety_is_stronger_than_indescribability} for the $n$-subtle ideal (for $n<\omega$), however the proof does not seem to generalize to the pre-Ramsey ideal.
\end{question}

As far as the author is aware, the following two questions are open. See Section \ref{section_splitting} above for more information.
\begin{question}
If $\kappa$ is subtle (and not weakly compact), is the subtle ideal on $\kappa$ nowhere $\kappa$-saturated?
\end{question}
\noindent Note that Abe \cite{MR2210149} proved that when $\lambda>\kappa$ the subtle ideal on $P_\kappa\lambda$ is not $\lambda$-saturated. However, the proof does not seem to give information about the subtle ideal on $\kappa$.

\begin{question}[Hellsten]
For $n<\omega$, if $\kappa$ is weakly $\Pi^1_n$-indescribable (and not inaccessible), is the weakly $\Pi^1_n$-indescribable ideal nowhere $\kappa$-saturated?
\end{question}

As pointed out by Hellsten, the techniques used in \cite{MR2653962} do not seem to provide an answer to the following.

\begin{question}[Hellsten]\label{question_hellsten}
Can the $\Pi^1_2$-indescribable ideal on $\kappa$ be $\kappa^+$-saturated?
\end{question}

The following question remains open, see Section \ref{section_weakly_compact_reflection} or \cite{MR3985624} for more details.

\begin{question}
If $\kappa$ is $\kappa^+$-weakly compact, is there a forcing extension in which $\kappa$ remains $\kappa^+$-weakly compact and there is a weakly compact subset of $\kappa$ with no weakly compact proper initial segment?
\end{question}

It is known \cite{CodyGitmanLambieHanson} that $\kappa$-strategically closed forcing cannot make ground model weakly compact sets become weakly compact in the extension, it is also known that this can fail for $\Pi^1_2$-indescribable sets, in other words, a set which is not $\Pi^1_2$-indescribable in the ground model can become so after $\kappa$-strategically closed forcing (see \cite{CodyGitmanLambieHanson}). The following related question concerning the $\Pi^1_2$-indescribable ideal remains open. See Theorem \ref{theorem_weakly_compact_ideal_after_forcing} above or \cite{MR4050036} for more details.

\begin{question}
Suppose $\kappa$ is $\Pi^1_2$-indescribable and $\P=\<(\P_\alpha,\dot{\Q}_\alpha)\st\alpha<\kappa\>\subseteq V_\kappa$ is a good Easton-support iteration of length $\kappa$ as in Theorem \ref{theorem_cody_sakai}. Is the $\Pi^1_2$-indescribable ideal of the extension by $\P$ equal to the ideal generated by the ground model $\Pi^1_2$-indescribable ideal?
\end{question}

As is the case for Question \ref{question_hellsten}, using current techniques, namely those used in \cite{CodyGitmanLambieHanson}, we seem to be unable to answer the following.

\begin{question}[\cite{CodyGitmanLambieHanson}]
From some large cardinal hypothesis on $\kappa$, can one force $\square_2(\kappa)$ to hold \emph{nontrivially}\footnote{See \cite{CodyGitmanLambieHanson} for more details.} while preserving the $\Pi^1_2$-indescribability of $\kappa$?
\end{question}


\section*{Acknowledgment}
The author would like to thank Sean Cox, Monroe Eskew, Victoria Gitman and Chris Lambie-Hanson for many helpful conversations regarding the topics of this article. The author also thanks the anonymous referee for the detailed review which greatly improved this article.




\end{document}